\documentclass[12pt,a4paper]{article}

\RequirePackage{etex}
\usepackage[utf8]{inputenc}
\usepackage{amsmath, amssymb, amsthm, amsfonts, mathrsfs}
\usepackage{bbm, commath}
\usepackage{multicol}
\usepackage{blindtext}
\usepackage{graphicx}
\usepackage{tikz-network}
\usepackage{subcaption}
\usetikzlibrary {positioning, patterns, calc}
\usetikzlibrary{graphs, graphs.standard, quotes}
\usepackage{authblk}
\usepackage{enumitem}
\usepackage[colorlinks=true,linkcolor=black,citecolor=black]{hyperref}
\usepackage{tabularx}
\usepackage{array}

\usepackage{natbib}
 \usepackage{setspace}
 \let\oldbibitem\bibitem
 \renewcommand{\bibitem}{\setstretch{1.2}\oldbibitem}

\newtheorem{thm}{Theorem}

\newtheorem{prop}{Proposition}

\newtheorem{cor}{Corollary}

\newtheorem{rem}{Remark}
\newtheorem{exa}{Example}

\def \Zl {{\mathbbm Z}}

\def \Rl {{\mathbbm R}}
\def \Cl {{\mathbbm C}}
\def \e {{\bf e}}
\def \u {{\bf u}}
\def \v {{\bf v}}
\def \f {{\bf f}}

\def \e {{\mathbf e}}

\newcommand{\ob}[1]{\left(#1\right)}
\newcommand{\cb}[1]{\left\lbrace #1\right\rbrace}
\newcommand{\tb}[1]{\left[#1\right]}
\newcommand{\mb}[1]{\left|#1\right|}

\title{\it Quantum walks on finite and bounded infinite graphs}
\author[1]{Chris Godsil \thanks{cgodsil@uwaterloo.ca}}
\author[2]{Steve Kirkland \thanks{stephen.kirkland@umanitoba.ca}}
\author[3]{Sarojini Mohapatra\thanks{mohapatrasarojini5@gmail.com}}
\author[2]{Hermie Monterde\thanks{monterdh@myumanitoba.ca}}
\author[3]{Hiranmoy Pal \thanks{palh@nitrkl.ac.in}}

\affil[1]{Department of Combinatorics \& Optimization, University of Waterloo, Waterloo, Ontario, Canada}
\affil[2]{Department of Mathematics, University of Manitoba, Winnipeg, MB, Canada R3T 2N2}
\affil[3]{National Institute of Technology Rourkela, India-769008  }

\date{\today}

\begin{document}
	
	\maketitle

	
\begin{abstract}
A weighted graph $G$ with countable vertex set is bounded if there is an upper bound on the maximum of the sum of absolute values of all edge weights incident to a vertex in $G$. In this paper, we prove a fundamental result on equitable partitions of bounded weighted graphs with twin subgraphs and use this fact to construct finite and bounded infinite graphs with pair and plus state transfer with the adjacency matrix as a Hamiltonian. We show that for each $k \ge 3$, (i) there are infinitely many connected unweighted graphs with maximum degree $k$ admitting pair state transfer at $\tau\in\{\frac{\pi}{\sqrt{2}},\frac{\pi}{2}\}$, and (ii) there are infinitely many signed graphs with exactly one negative edge weight and whose underlying unweighted graphs have maximum degree $k$ admitting plus state transfer at $\tau\in\{\frac{\pi}{\sqrt{2}},\frac{\pi}{2}\}$. Parallel results are proven for perfect state transfer between a plus state and a pair state, and for the existence of sedentary pair and plus states. We further prove that almost all connected unweighted finite planar graphs admit pair state transfer at $\tau\in\{\frac{\pi}{\sqrt{2}},\frac{\pi}{2}\}$, and almost all connected unweighted finite planar graphs can be assigned a single negative edge weight resulting in plus state transfer, or perfect state transfer between a plus state and a pair state, at $\tau\in\{\frac{\pi}{\sqrt{2}},\frac{\pi}{2}\}$. Analogous results are shown to hold for unweighted finite trees. Using blow-up graphs, Cayley graphs and graphs with tails, we construct new infinite families of (finite and infinite) unweighted graphs and signed graphs admitting pair or plus state transfer.\\

\noindent{\it Keywords:} continuous quantum walk, equitable partition, signed graph, infinite graph, perfect state transfer, pair states.\\
\noindent{\it MSC: 81P45, 05C50, 15A16, 15A60.}
\end{abstract}

	
  \section{Introduction}

Accuracy matters. This statement applies in many settings,  but in particular, it is true in the context of transferring quantum information in a network of interacting qubits. Such a network can be modeled by a graph, where the vertices represent the qubits, and the edges represent interactions between pairs of qubits. Sparse graphs are of particular interest here, as the realization of a qubit network has a cost associated with each edge in the graph. Consequently, trees, planar graphs, and graphs of small maximum degree are natural candidates to consider in the context of networks of qubits. 

The quantum walk on a graph $G$ is governed by the \textit{transition matrix} $\exp{(itA)}$, where $A$ is the adjacency matrix of $G$. If for unit vectors $\u,\v\in\Cl^n$ we have $\exp{(i\tau A)}\u = \gamma \v$ for some $\gamma\in\Cl$, then \textit{perfect state transfer} (PST) occurs from $\u$ to $\v$ at $\tau$ \cite{godsil2025perfect}. This corresponds to accurate communication of the pure quantum state associated with $\u$ to the pure quantum state associated with $\v$, and is a highly desirable property. Much of the literature in quantum walks focuses on vertex PST on finite graphs, i.e.\ when $\u$ and $\v$ are standard basis vectors and $G$ has finite vertex set (see \cite{god1} for a survey and \cite{pal9,cou7,kirkland2023quantum,kirk2,monterde2025laplacian,pal25} for recent work). However, it seems that vertex PST on finite graphs is rare. Godsil showed that for any integer $k>0$, there is a finite number of connected unweighted finite graphs of maximum degree $k$ with vertex PST \cite[Corollary 6.2]{god2}. Motivated by this result, there is an emerging literature on extensions of vertex PST, such as PST between pair states (unit vectors of the form $\frac{1}{\sqrt{2}}(\e_u-\e_v)$), also known as \textit{pair PST}, and PST between plus states (unit vectors of the form $\frac{1}{\sqrt{2}}(\e_u+\e_v)$), also known as \textit{plus PST} \cite{che1,kim,ojha25,pal10}. 
For infinite graphs, only quantum walks on graphs attached to a finite number of infinite paths have been investigated \cite{bernard2025quantum,childs2009universal,childs2013universal,farhi2007quantum}. The growing interest in pair and plus PST and the scant work on quantum walks on infinite graphs prompts us in this paper to investigate perfect state transfer on weighted graphs with countable vertex sets that are bounded (i.e., there is a bound on the maximum of the sum of absolute values of edge weights incident to a vertex) with focus on pair PST, plus PST, and PST between a pair state and a plus state. 

We now outline our main contributions in this paper. Using equitable partitions, we establish a connection between plus PST in a bounded graph and vertex PST in a quotient graph. We also relate the existence of pair PST in a bounded graph with (edge-perturbed) twin subgraphs and vertex PST in the graph induced by the vertices of the subgraphs in question. In contrast to the rarity of vertex PST in finite graphs, we show that for each $k\geq 3$, there are infinitely many connected unweighted graphs with maximum degree $k$ admitting pair PST at time $\tau\in\{\frac{\pi}{2},\frac{\pi}{\sqrt{2}}\}$. We also prove a stronger result for planar graphs (resp., trees), which states that almost all connected unweighted finite planar graphs (resp., trees) admit pair PST at time $\tau$. Consequently, for each $k\geq 2$, there are more connected unweighted finite graphs (resp., planar graphs, trees) with maximum degree $k$ admitting pair PST than those admitting vertex PST. This suggests that PST between entangled vertex states (represented by pair states) is a more common phenomenon than vertex PST. Moreover, we prove that for each $k\geq 3$, there are infinitely many signed graphs with exactly one negative edge weight and whose underlying unweighted graphs have maximum degree $k$ admitting plus PST and PST between a pair state and a plus state at time $\tau\in\{\frac{\pi}{\sqrt{2}},\frac{\pi}{2}\}$. For planar graphs (resp., trees),  we establish a stronger result which states that almost all unweighted finite planar graphs (resp., trees) can be assigned a single negative edge weight to produce plus PST or PST between a plus state and a pair state at time $\tau$. Finally, using blow-up graphs, Cayley graphs, and graphs with tails, we supply new infinite families of unweighted graphs admitting pair and plus PST. The constructions we provide apply to both finite and infinite graphs, leading to new examples of pair and plus PST in infinite graphs that do not necessarily have tails. This observation is in stark contrast to vertex PST, which only occurs in finite graphs \cite{godd}. 


\subsection{Graphs}

Let $G$ be an undirected weighted graph with vertex set $V(G)$ and edge set $E(G)$. We allow our graphs to be infinite with countable vertex set and non-zero real edge weights. Define a weight function $w: V(G)\times V(G)\to \mathbb{R}$ such that $w(a,b)=w(b,a)$ for all $\{a,b\}\in E(G)$, and $w(a,b)=0$ if and only if $\{a,b\}\notin E(G)$. If $w(a,b)\neq 0$, then $w(a,b)$ is called the edge weight of $\{a,b\}$. We say that $G$ is a \textit{positively weighted graph} (resp., negatively weighted graph) if all edge weights in $G$ are positive (resp., negative). If $w(a,b)=1$ for all $\{a,b\}\in E(G)$, then $G$ is an \textit{unweighted graph}. We say that $G$ is a \textit{weighted signed graph} if $G$ has an edge with positive weight, and another one with a negative weight. In particular, if $w(a,b)=\pm 1$ for all $\{a,b\}\in E(G)$, then $G$ is an \textit{unweighted signed graph}. Unless otherwise specified, a graph $G$ denotes an undirected graph whose edge weights are arbitrary nonzero real numbers. We denote the unweighted path, cycle, and complete graphs on $n$ vertices by $P_n$, $C_n$ and $K_n$, respectively. We let $N_G(a)$ be the set of neighbours of vertex $a$ in $G$ and we refer to $\{w(a,b):b\in N_G(a)\}$ as the \textit{sequence of edge weights} of $a$. The \textit{degree} of vertex $a$ of $G$ is the sum $\sum_{b \in N_G(a)} w(a,b)$, while its \textit{absolute degree} is the sum $\sum_{b \in N_G(a)} |w(a,b)|$. If the absolute degree of vertex $a$ is finite, then so is its degree. 

A graph $G$ is \textit{locally finite} if for each vertex $a$ of $G$, the sequence of edge weights of $a$ is square summable, i.e., $\sum_{b \in N_G(a)} |w(a,b)|^2<\infty$. The \textit{adjacency matrix} $A(G)$ of a locally finite graph $G$ is a real symmetric matrix indexed by the vertices of $G$ such that $A(G)_{a,b}=w(a,b).$ In this case, we may consider $A(G):\mathbb{C}^{|V(G)|}\rightarrow \mathbb{C}^{|V(G)|}$ as a linear operator where the $u$--th entry of $A(G)\f$ for each $\f\in \mathbb{C}^{|V(G)|}$ is given by
\begin{center}
$A(G)\f(u)=\displaystyle\sum_{v\in N_G(u)}(\f^T\e_v)w(u,v)$.
\end{center}
As $G$ is locally finite, $|A(G)\f(u)|$ is bounded above for all $u\in V(G)$ and $\f\in \mathbb{C}^{|V(G)|}$ by the Cauchy-Schwarz inequality.
A locally finite graph $G$ is \textit{bounded} if the maximum absolute degree  in $G$ is bounded. Note that any finite weighted graph is bounded. Moreover, the infinite star with sequence of edge weights given by $\{\frac{(-1)^n}{n}:n\in\Zl^+\}$ is locally finite but not bounded, since the central vertex has unbounded absolute degree (but finite degree).

Recall that $\|A(G)\|=\sup_{\|\f\|=1}\|A(G)\f\|$. If $A(G)$ is a bounded linear operator, 
i.e., $\|A(G)\|<\infty$, then $G$ is locally finite since for each vertex $a$ of $G$,
\begin{equation}
\label{locfin}
\sum_{b \in N_G(a)} |w(a,b)|^2=\|A(G)\e_a\|^2\leq \|A(G)\|^2<\infty.
\end{equation}

The following result is relevant throughout this work.

\begin{thm}
\label{infg}
Let $G$ be a locally finite graph. If $G$ is bounded and $\mathcal{M}$ is the maximum absolute degree in $G$, then $A(G)$ is a bounded linear operator on the Hilbert space $\ell^2(\mathbb{Z}^+)$ of square-summable functions on $V(G)$ and $\|A(G)\|\leq\mathcal{M}$. The converse holds whenever each edge weight in $G$ has absolute value at least one. 
\end{thm}

\begin{proof}
The first and second statements follow resp.\ from \cite[Theorem 6.12-A]{taylor1958introduction} and (\ref{locfin}).
\end{proof}

If $G$ is an unweighted graph, then Theorem \ref{infg} recovers a result of Mohar \cite[Theorem 3.2]{mohar1982spectrum} which states that $G$ is bounded if and only if $A(G)$ is a bounded linear operator. 


\subsection{Quantum walks}

A pure state $\u$ in $G$ is a unit vector in $\mathbb{C}^{|V(G)|}.$ Pure states in $G$ of the form $\e_a$, $\frac{1}{\sqrt{2}}\ob{\e_a-\e_b}$ and $\frac{1}{\sqrt{2}}\ob{\e_a+\e_b}$ are called vertex states, pair states, and plus states, respectively. Often, we drop the normalization factor for a pair or a plus state for the sake of brevity.

A \textit{quantum walk} on a bounded graph $G$ is determined by a one-parameter family of unitary matrices given by 
\begin{equation*}
U_G(t)=\exp{(itA(G))}=\displaystyle\sum_{k=0}^\infty\frac{(it)^kA(G)^k}{k!},\quad t\in\Rl
\end{equation*}
where $i^2=-1.$ Since $G$ is bounded, $A(G)$ is a bounded linear operator on $\ell^2(\mathbb{Z}^+)$ by Theorem \ref{infg}, and so the above sum converges. The matrix $U_G(t)$ is the \textit{transition matrix} of $G$. Note that we may ignore the case when $G$ is negatively weighted, since $G$ determines the same quantum walk as the positively weighted graph $H$ where $A(H)=-A(G)$.


Let $\u$ and $\v$ be pure states. 
A graph $G$ is said to have \textit{perfect state transfer} (PST) from $\u$ to $\v$ at $\tau \in \mathbb{R}$ if
\begin{center}
$U_G(\tau)\u=\gamma\v\quad \text{for some}~\gamma\in\mathbb{C}.$
\end{center}
If $\u$ and $\v$ are linearly dependent, then $\u$ is \textit{periodic} in $G$ at $\tau$, and we may write the above equation as $U_G(\tau)\u=\gamma\u$ for some $\gamma\in\mathbb{C}.$
We say that $G$ admits \textit{pretty good state transfer} (PGST) from $\u$ to $\v$ relative to the sequence $\{\tau_k\} \subseteq \mathbb{R}$ if
\begin{center}
$\displaystyle\lim_{k\rightarrow\infty}U_G(\tau_k)\u=\gamma\v\quad\text{for some}~\gamma\in\mathbb{C}.$
\end{center}
If $\u$ and $\v$ are real, then there is PST (resp., PGST) from $\u$ to $\v$ if and only if there is PST (resp., PGST) from $\v$ to $\u$ at the same time (resp., sequence of times). In this case, we simply say that there is PST (resp., PGST) between $\u$ and $\v$. In particular, if $\u$ and $\v$ are pair states (resp., plus states), then we also say pair PST or pair state transfer (resp., plus PST or plus state transfer) in lieu of PST between the pair states (resp., plus states) $\u$ and $\v$. Lastly, a pure state $\u$ is \textit{$C$-sedentary} in $G$ if
\begin{center}
$\displaystyle\inf_{t>0}~\mb{\u^*U_G(t)\u}\geq C$
\end{center}
for some constant $0 < C \leq 1$. If $C$ is not important, then we say that $\u$ is sedentary.

In general, one may define a quantum walk on $G$ using any real symmetric matrix $M$ such that $M_{u,v}=0$ if and only if there is no edge between $u$ and $v$. Such a matrix is called the \textit{Hamiltonian} of the quantum walk, and is an infinite matrix whenever $G$ is infinite. Hence, unless otherwise stated, we assume that the Hamiltonian taken is $A(G)$ whenever we say PST (resp., PGST) from $\u$ to $\v$ in $G$, periodicity of $\u$ in $G$ and sedentariness of $\u$ in $G$.  However, if PST (resp., PGST) occurs from $\u$ to $\v$ in $G$ relative to a Hamiltonian $M$ that is different from $A(G)$, then we simply say PST (resp., PGST) from $\u$ to $\v$ relative to $M$, periodicity of $\u$ relative to $M$, and sedentariness of $\u$ relative to $M$, respectively.


\subsection{Equitable partitions}

A \textit{partition} $\Pi$ of a graph $G$ is a partition $\bigcup_j V_j$ of $V(G)$, where each $V_j$ is a finite set. The $V_j$'s are called the \textit{cells} of $\Pi$. The \textit{normalized characteristic matrix} $\mathcal{C}$ associated with $\Pi$ is a matrix whose rows are indexed by the vertices of $G$ and $j$--th column is the normalized characteristic vector of the cell $V_j.$ If $I$ is the identity matrix, $J_k$ is the $k\times k$ all-ones matrix and $\displaystyle\bigoplus_{k} M_k$ is the direct sum of matrices $M_k$, then
\begin{center}
$\mathcal{C}^T\mathcal{C}=I\quad \text{and}\quad \displaystyle\mathcal{C}\mathcal{C}^T=\bigoplus_{k}\frac{1}{|V_k|}J_{|V_k|}.$
\end{center}
If $G$ is finite, then $\mathcal{C}$ is a $|V(G)|\times d$ matrix, where $d$ is the number of cells in $\Pi$.

A \textit{partition} $\Pi$ of a bounded graph $G$ is \textit{equitable} if for each $ a\in V_j$, the sum $\sum_{b \in V_k} w(a,b)$ is a constant 
$c_{jk}\in \Rl$ for all $j$. In the unweighted case, this means that the number of neighbours in $V_k$ of a vertex in $ V_j$ depends only on the choice of $j$ and $k$.
If $\Pi$ is equitable, then
\begin{center}
$c_{j,k}|V_j|=c_{k,j}|V_k|$,
\end{center}
and so $c_{j,k}$ and $c_{k,j}$ are either both positive, both negative or both zero. An equitable partition $\Pi$ of $G$ gives rise to a \textit{(symmetrized) quotient graph} $G/\Pi$, which is the undirected graph with adjacency matrix $A(G/\Pi)$ defined entry-wise as follows:
\[A(G/\Pi)_{j,k}= \left\{ \begin{array}{rcl}
 \sqrt{c_{jk}c_{kj}}, & \mbox{if} & c_{jk}\geq 0, \\ 
 -\sqrt{c_{jk}c_{kj}}, & \mbox{if} & c_{jk}<0.
 \end{array}\right.\] 
It is known that $A(G)\mathcal{C} = \mathcal{C}A(G/\Pi)$ and the matrices $A(G)$ and $\mathcal{C}\mathcal{C}^T=\bigoplus_{k} \frac{1}{|V_k|}J_{|V_k|}$ commute \cite[Lemma 4.2]{god2}. Since $G$ is bounded, Theorem \ref{infg} implies that the power series expansion of $\exp{(itA(G))}$ converges. Thus,
\begin{equation}
\label{G/Pi}
U_G(t)\mathcal{C}=\mathcal{C}U_{G/\Pi}(t)\quad \text{for all $t\in\Rl$}.
\end{equation}
For more on equitable partitions, we refer the interested reader to \cite{god0}.

\section{Twin subgraphs}\label{2s_3}

Let $X_1$ and $X_2$ be graphs with isomorphism $f: X_1\to X_2$ satisfying $w\ob{f(a),f(b)}=w(a,b)$ for all pairs of vertices $a$ and $b$ in $X_1.$ Consider a connected graph $G$ where the disjoint union $X_1\cup X_2$ of $X_1$ and $X_2$ appears as an induced subgraph of $G$ where $w\ob{f(a),y}=w(a,y),$ for all $a\in V\ob{X_1}$ and $y\in V(G)\setminus V\ob{X_1\cup X_2}.$ We say that $X_1$ and $X_2$ are \textit{twin subgraphs} of $G$ with $\hat{f}$ as the automorphism of $G$ that switches each vertex of $X_1$ to its $f$-isomorphic copy in $X_2$ and fixes all other vertices of $G$.

\begin{figure}
\begin{multicols}{2}
\begin{center}
\begin{tikzpicture}[scale=.5,auto=left]
                       \tikzstyle{every node}=[circle, thick, fill=white, scale=0.6]
                       
		        \node[draw] (1) at (1.5,0) {$3$};		        
		        \node[draw,minimum size=0.7cm, inner sep=0 pt] (2) at (-2.8, 1) {$1$};
		        \node[draw,minimum size=0.7cm, inner sep=0 pt] (3) at (-1.3, 2) {$2$};
		        \node[draw,minimum size=0.7cm, inner sep=0 pt] (4) at (-2.8, -1) {$5$};
		        \node[draw,minimum size=0.7cm, inner sep=0 pt] (5) at (-1.3, -2) {$4$};		       
	
				\node at (-4.5, 2) {$X_1$};
				\node at (-4.5,-2) {$X_2$};
				\node at (6.3,0.5) {$H$};
				
				\draw[dotted] (-2,1.5) ellipse (2 cm and 1.3 cm);
				\draw[dotted] (-2,-1.5) ellipse (2 cm and 1.3 cm);
				\draw[dotted] (3.2,0.4) circle (2.5 cm);
								
				\draw [thick, black!70] (1)--(3)--(2);
				\draw [thick, black!70] (5)--(4);
                \draw [thick, black!70] (1)--(5);

				\draw[thick, black!70] (1)..controls (3,4) and (8,0)..(1);

				\end{tikzpicture}
                
\begin{tikzpicture}[scale=.5,auto=left]
                       \tikzstyle{every node}=[circle, thick, fill=white, scale=0.6]
                       
		        \node[draw] (1) at (1.5,0) {$3$};		        
		        \node[draw,minimum size=0.7cm, inner sep=0 pt] (2) at (-2.8, 1) {$1$};
		        \node[draw,minimum size=0.7cm, inner sep=0 pt] (3) at (-1.3, 2) {$2$};
		        \node[draw,minimum size=0.7cm, inner sep=0 pt] (4) at (-2.8, -1) {$5$};
		        \node[draw,minimum size=0.7cm, inner sep=0 pt] (5) at (-1.3, -2) {$4$};		       
	
				\node at (-4.5, 2) {$X_1$};
				\node at (-4.5,-2) {$X_2$};
				\node at (6.3,0.5) {$H$};
				
				\draw[dotted] (-2,1.5) ellipse (2 cm and 1.3 cm);
				\draw[dotted] (-2,-1.5) ellipse (2 cm and 1.3 cm);
				\draw[dotted] (3.2,0.4) circle (2.5 cm);
								
				\draw [thick, black!70] (1)--(3)--(2);
				\draw [thick, black!70] (5)--(4)--(3);
                \draw [thick, black!70] (1)--(5)--(2);

				\draw[thick, black!70] (1)..controls (3,4) and (8,0)..(1);

				\end{tikzpicture}
				
\end{center}
\end{multicols}
\caption{\label{fi} A graph with $P_2$ as twin subgraphs (left) and an edge-perturbed version (right)}
\end{figure}
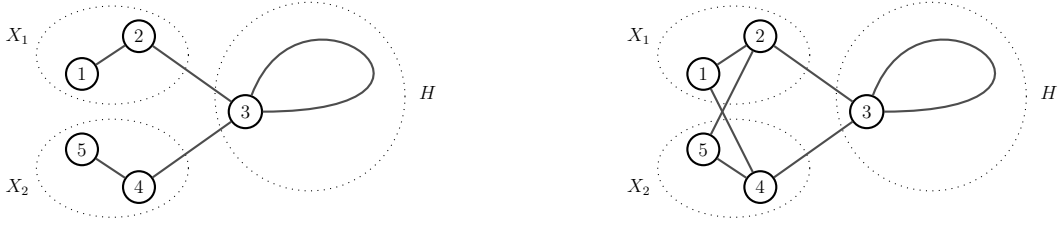

Let $X_1$ and $X_2$ be twin subgraphs of a graph $G$. Create $G'$ by adding edges in $G$ between $X_1$ and $X_2$ in a way that preserves the automorphism $\hat{f}$ of $G$. That is, $A(G')_{a,\hat{f}(b)}=A(G')_{\hat{f}(a),b}$. In this case, we call $X_1$ and $X_2$ \textit{edge-perturbed twin subgraphs} in $G'$, and $\hat{f}$ is an automorphism of both $G'$ and $G$ that switches each vertex of $X_1$ to its $f$-isomorphic copy in $X_2$ and fixes all other vertices of $G'$.

The following generalizes a fundamental fact on equitable partitions \cite[Lemma 4.2]{god2}. 

\begin{thm}
\label{godgen}
Let $G$ be a bounded graph with edge-perturbed twin subgraphs $X_1$ and $X_2$ with $A=A(G)$, and $\Pi$ be a partition of $G$ where $\{a,\hat{f}(a)\}$ is a cell for each $a\in V (X_1)$. Let $Q=[\mathcal{B}~~\mathcal{C}]$ be the matrix where the columns of $\mathcal{B}$ consist of all vectors from the set $\{\frac{1}{\sqrt2}(\e_a-\e_{\hat{f}(a)}):a\in V(X_1)\}$ and $\mathcal{C}$ is the normalized characteristic matrix of $\Pi.$ The following are equivalent:
\begin{enumerate}
\item $\Pi$ is equitable.
\item The column space of $Q$ is $A$-invariant.
\item There is a matrix $B$ such that $AQ=QB$.
\item The matrices $A$ and $QQ^T=\displaystyle I_{2|V(X_1)|}\oplus \left(\bigoplus_{k>|V(X_1)|} \frac{1}{|V_k|}J_{|V_k|}\right)$ commute.
\end{enumerate}
Moreover, if such a matrix $B$ exists, then
\begin{equation}
\label{B}
B:=\begin{bmatrix}
     A(X_1)-A' & \mathbf{0}\\
     \mathbf{0} & A(G/\Pi)
     \end{bmatrix},
\end{equation}
and $A'$ is a symmetric matrix whose rows and columns are indexed by the vertices of $X_1$ and $X_2,$ respectively, and $(A')_{u,v}\neq 0$ if and only if $u$ and $v$ are adjacent in $G$.
\end{thm}
\begin{proof}
We first prove $1\Leftrightarrow 3$. Assume $\Pi$ is equitable so that $A\mathcal{C}=\mathcal{C}A(G/\Pi)$. As $X_1$ and $X_2$ are edge-perturbed twin subgraphs in $G$ and $\{a,\hat{f}(a)\}$ is a cell for all $a\in V (X_1)$, we get $A\mathcal{B}=\mathcal{B}(A(X_1)-A')$. Combining the two preceding equations gives us
\begin{equation*}
AQ=[A\mathcal{B}~~A\mathcal{C}]=[\mathcal{B}(A(X_1)-A')~~\mathcal{C}A(G/\Pi)]=[\mathcal{B}~~\mathcal{C}]\begin{bmatrix}
     A(X_1)-A' & \mathbf{0}\\
     \mathbf{0} & A(G/\Pi)
     \end{bmatrix}=QB.
\end{equation*}
Conversely, suppose $AQ=QB$. 
Let $\e_a$ be the characteristic vector of vertex $a$ in $G$, and let $\check{\e}_a$ be the characteristic vector of vertex $a$ such that $Q\check{\e}_a=\frac{1}{\sqrt{2}}(\e_a-\e_{\hat{f}(a)})$ if $a\in V(X_1)$, and $Q\check{\e}_a$ is the normalized characteristic vector of the cell containing the vertex $a$, otherwise. If $a\in V(G)\backslash V(X_1)$ belong to $V_j$, then $\e_a^TQ=\e_b^TQ=\frac{1}{\sqrt{|V_j|}}\check{\e}_j^T$. Thus, if $k\in V(G)\backslash V(X_1)$, then 
$\e_a^T(AQ)\check{\e}_k=\e_a^T(QB)\check{\e}_k=\frac{1}{\sqrt{|V_j|}}\check{\e}_j^TB\check{\e}_k=\frac{1}{\sqrt{|V_k|}}\displaystyle\sum_{v\in V_k}A_{a,v}$. Hence,
\begin{center}
$\displaystyle\sum_{v\in V_k}A_{a,v}= \sqrt\frac{|V_k|}{|V_j|}\check{\e}_j^TB\check{\e}_k=c_{j,k}.$
\end{center}
As the above sum is independent of the choice of vertex in $V_j$, $\Pi$ is equitable. 


Next, we prove $2\Leftrightarrow 3$. If 3 holds, then $AQ\check{\e}_u=QB\check{\e}_u$ for all $u$. Since $Q\check{\e}_u$ and $QB\check{\e}_u$ belong to the column space of $Q$, we obtain 2. The converse is straightforward.

Lastly, we prove $3\Leftrightarrow 4$.
If $AQ=QB$, then $Q^TA=BQ^T$ since $A$ and $B$ are symmetric. So, $AQQ^T=QBQ^T=QQ^TA$. Conversely, if $AQQ^T=QQ^TA$, then the fact $Q^TQ=I$ yields $Q^TAQQ^T=Q^TA$. If $B:=Q^TAQ$, then $BQ^T=Q^TA$, and because $A$ and $B$ are symmetric, we get $AQ=QB$. Finally, 
using the fact that $\mathcal{C}\mathcal{C}^T=\displaystyle\bigoplus_{k}\frac{1}{|V_k|}J_{|V_k|}$ gives us
\begin{equation*}
QQ^T
=\mathcal{B}\mathcal{B}^T+\mathcal{C}\mathcal{C}^T=\left(\frac{1}{2}\bigoplus_{|V(X_1)|}\begin{bmatrix}1&-1\\-1&1\end{bmatrix}\oplus O\right)+\left(\bigoplus_{|V(X_1)|}\frac{1}{2}J_2\oplus \bigoplus_{k>|V(X_1)|} \frac{1}{|V_k|}J_{|V_k|}\right),
\end{equation*}
where $O$ is the zero matrix of the appropriate size. This yields the form of $QQ^T$ in 4.

Combining the above cases proves the equivalence of 1-4. Finally, if a matrix $B$ in 3 exists, then $\Pi$ is equitable. As $\mathcal{B}^T\mathcal{B}=I,$ we get $\mathcal{B}^TA\mathcal{B}=A(X_1)-A'$ and $\mathcal{C}^TA\mathcal{C}=A(G/\Pi).$ Now, if $a\in V(X_1)$ and $b\in V(G)\backslash V(X_1)$, then $Q\check{\e}_a= \frac{1}{\sqrt{2}}\ob{\e_a-\e_{\hat{f}(a)}}$ and $Q\check{\e}_b$ is a characteristic vector of $\Pi$. As $X_1$ and $X_2$ are edge-perturbed twin subgraphs in $G$,
\begin{center}
$\check{\e}_a^T(Q^TAQ)\check{\e}_b=(Q\check{\e}_a)^TA(Q\check{\e}_b)=\frac{1}{\sqrt{2}}\ob{\e_a-\e_{\hat{f}(a)}}^TA(Q\check{\e}_b)=0.$
\end{center} 
Therefore, $B=Q^TAQ=\begin{bmatrix} \mathcal{B}^TA\mathcal{B} & \mathcal{B}^TA\mathcal{C}\\ \mathcal{C}^TA\mathcal{B} & \mathcal{C}^TA\mathcal{C}\end{bmatrix}$, and because $\mathcal{B}^TA\mathcal{B}=A(X_1)-A'$ and $\mathcal{C}^TA\mathcal{C}=A(G/\Pi)$, we obtain the form of $B.$
\end{proof}

Note that if $G$ is finite, then $Q$ is a $|V(G)|\times(|V(X_1)|+d)$ matrix.

Our next result is useful throughout this work.

\begin{thm}\label{mainthm}
Let $G$ be a bounded graph with edge-perturbed twin subgraphs $X_1$ and $X_2$, and let $U'(t)=\exp{\ob{it\ob{A\ob{X_1}-A'}}}$. If $G$ admits an equitable partition $\Pi$ where $\{a, \hat{f}(a)\}$ forms a cell for every $a\in V(X_1)$, then
\[Q^T U_G(t)Q=\begin{bmatrix}U'(t) & \mathbf{0}\\ \mathbf{0} & U_{G/\Pi}(t) \end{bmatrix},\quad \text{for all $t\in\Rl$}.\]
\end{thm}

\begin{proof}
From Theorem \ref{godgen}(3), it follows that $AQ=QB$ where $B=\begin{bmatrix}
     A(X_1)-A' & \mathbf{0}\\
     \mathbf{0} & A(G/\Pi)
     \end{bmatrix}$.
     Thus, $Q^TA(G)Q=B.$ From Theorem \ref{godgen}(4), $QQ^T$ commutes with $A(G)$, and so
\begin{equation}
\label{eqpart}
Q^TA(G)^kQ=(Q^TA(G)Q)^k=B^k=\begin{bmatrix}
     (A(X_1)-A')^k & \mathbf{0}\\
     \mathbf{0} & (A(G/\Pi))^k
     \end{bmatrix}
\end{equation}
for all integers $k\geq 0$. Finally, since $G$ is bounded, Theorem \ref{infg} implies that power series expansion of $\exp{(itA(G))}$ converges, in which case (\ref{eqpart}) gives us
\begin{equation}
\label{aa}
\begin{split}
Q^TU_G(t)Q &=\sum_{k=0}^\infty\frac{(it)^k}{k!}(Q^TA(G)^kQ)=\sum_{k=0}^\infty\frac{(it)^k}{k!}B^k=\begin{bmatrix}
     U'(t) & \mathbf{0}\\
     \mathbf{0} & U_{G/\Pi}(t)
     \end{bmatrix},
\end{split}
\end{equation}
which establishes the desired result.
\end{proof}

If $G$ is not bounded, then it is not clear how we can simplify $Q^TU_G(t)Q$ in (\ref{aa}).

\begin{rem}
If the graph $G$ in Theorem \ref{mainthm} has an equitable partition $\Pi$ but has no (edge-perturbed) twin subgraphs, then the vectors in $\mathcal{B}$ are absent as columns of $Q$. In this case, $Q=\mathcal{C}$ and so we may view the result in Theorem \ref{mainthm} as an extension of (\ref{G/Pi}).
\end{rem}

\begin{rem}
\label{A'}
In Theorem \ref{mainthm}, if $X_1$ and $X_2$ are twin subgraphs and $G$ has no edges between $X_1$ and $X_2$, then $A'=0$, so $U'(t)$ is equal to $U_{X_1}(t)$, the transition matrix of $X_1$.   \end{rem}

Since $Q^TQ=I$, the form of $Q^TU_G(t)Q$ in Theorem \ref{mainthm} yields the following result.

\begin{cor}
\label{maincor}
With the assumption in Theorem \ref{mainthm}, the following hold.
\begin{enumerate}
\item Perfect state transfer occurs from $\u$ to $\v$ relative to $A(X_1)-A'$ if and only if perfect state transfer occurs from $\mathcal{B}\u$ to $\mathcal{B}\v$ in $G$ at the same time.
\item The vector $\u$ is periodic relative to $A(X_1)-A'$ if and only if $\mathcal{B}\u$ is periodic in $G$ at the same time.
\item Perfect state transfer occurs from $\u$ to $\v$ in $G/\Pi$ if and only if perfect state transfer occurs from $\mathcal{C}\u$ to $\mathcal{C}\v$ in $G$ at the same time.
\item The vector $\u$ is periodic in $G/\Pi$ if and only if $\mathcal{C}\u$ is periodic in $G$ at the same time.
\end{enumerate}
\end{cor}

\begin{rem}
Corollary \ref{maincor}(1,3) hold if we replace `perfect state transfer' with `pretty good state transfer', while Corollary \ref{maincor}(2,4) hold if we replace `periodic' with `sedentary'. 
\end{rem}

\section{Signed graphs}\label{sec:sgned}

Here, we establish a connection between perfect state transfer in a positively weighted graph $G$ and in a weighted signed graph $\tilde{G}$ whenever $A(G)$ and $\pm A(\tilde{G})$ are similar.

\begin{prop}
\label{prop}
Let $G$ be a positively weighted graph and $D$ be an involutory matrix such that $DA(G)D=\delta A(\tilde{G})$ where $\delta=\pm 1$. Then $G$ has perfect state transfer from $\u$ to $\v$ at $\tau$ if and only if $\tilde{G}$ has perfect state transfer from $D\u$ to $D\v$ at time $\delta\tau$.
\end{prop}

\begin{proof}
Since $DA(G)D=\delta A(\tilde{G})$, we have $U_{G}(t)=DU_{\tilde{G}}(\delta t)D$. If perfect state transfer occurs from $\u$ to $\v$ in $G$ at $\tau$, then $DU_{\tilde{G}}(\delta \tau)D\u=U_{G}(\tau)\u=\gamma\v$ for some $\gamma\in\Cl$. Since $D^2=I$, we obtain $U_{\tilde{G}}(\delta\tau)(D\u)=\gamma(D\v)$. The converse is straightforward.
\end{proof}

If $G$ is a positively weighted graph and $D$ is a diagonal matrix with $\pm 1$ entries with at least one entry equal to $-1$, then $\tilde{G}$ is a weighted signed graph. Additionally, if $G$ is unweighted, then the preceding statement implies that $\tilde{G}$ is an unweighted signed graph.

A weighted signed graph $\tilde{G}$ is \textit{balanced} (resp., \textit{anti-balanced}) relative to a positively weighted graph $G$ if $DA(G)D=A(\tilde{G})$ (resp., $DA(G)D=-A(\tilde{G})$) for some diagonal matrix $D$ with $\pm 1$ entries. The following is immediate from Proposition \ref{prop}.

\begin{cor}
\label{bal}
A positively weighted graph has perfect state transfer from $\u$ to $\v$ at $\tau$ if and only if the corresponding balanced (resp., anti-balanced) weighted signed graph has perfect  state transfer from $D\u$ to $D\v$ at $\tau$ (resp., $-\tau$).
\end{cor}

Using positively weighted graphs with pair PST (resp., plus PST), we generate weighted signed graphs with plus PST (resp., pair PST) or PST between a plus state and a pair state. In what follows, $w(a,b)$ and $\tilde{w}(a,b)$ denote the weights of the edge $\{a,b\}$ in $G$ and $\tilde{G}$, respectively. For a diagonal matrix $D$, we let $D_a$ denote the $a$--th diagonal entry of $D$.

\begin{cor}
\label{cor}
Suppose $G$ is a positively weighted graph with perfect state transfer between $\frac{1}{\sqrt2}(\e_a-\e_b)$ and $\frac{1}{\sqrt2}(\e_c-\e_d)$ at $\tau$. The following hold.
\begin{enumerate}[leftmargin=*]
\item Let $\tilde{G}$ be the weighted signed graph with $\tilde{w}\ob{b,x}=-w(b,x)$ for all $x\in N_G(b)$.
\begin{enumerate}[leftmargin=*]
\item If $b\neq c,d$, then $\tilde{G}$ admits perfect state transfer between $\frac{1}{\sqrt2}(\e_a+\e_b)$ and $\frac{1}{\sqrt2}(\e_c-\e_d)$. 
\item Otherwise, $\tilde{G}$ admits perfect state transfer between $\frac{1}{\sqrt2}(\e_a+\e_b)$ and $\frac{1}{\sqrt2}(\e_c+\e_d).$  
\end{enumerate}
\item Let $\tilde{G}$ be the weighted signed graph with $\tilde{w}\ob{b,x}=-w\ob{b,x}$ and $\tilde{w}\ob{d,y}=-w(d,y)$ for all $x\in N_G(b)\backslash \cb{d}$ and for all $y\in N_G(d)\backslash \cb{b}$.
\begin{enumerate}
\item If $a=d,$ then $\tilde{G}$ admits perfect state transfer between $\frac{1}{\sqrt2}(\e_b-\e_d)$ and $\frac{1}{\sqrt2}(\e_c+\e_d).$
\item If $b=c,$ then $\tilde{G}$ admits perfect state transfer between $\frac{1}{\sqrt2}(\e_a+\e_b)$ and $\frac{1}{\sqrt2}(\e_d-\e_b).$
\item Otherwise, $\tilde{G}$ admits perfect state transfer between $\frac{1}{\sqrt2}(\e_a+\e_b)$ and $\frac{1}{\sqrt2}(\e_c+\e_d)$. 
\end{enumerate}
   \end{enumerate}
Furthermore, perfect state transfer occurs in 1 and 2 at $\tau$.
 \end{cor}

\begin{proof}
To prove 1, note that $\tilde{G}$ is balanced relative to $G$ since $DA(G)D=A(\tilde{G})$, where $D$ is the diagonal matrix of $\pm 1$'s with $D_j=-D_b=1$ for $j\in\{a,c,d\}$. 
Since $G$ has PST between $\u=\frac{1}{\sqrt2}(\e_a-\e_b)$ and $\v=\frac{1}{\sqrt2}(\e_c-\e_d)$, Corollary \ref{bal} yields PST in $\tilde{G}$ between $D\u=\frac{1}{\sqrt2}(\e_a+\e_b)$ and $D\v=\frac{1}{\sqrt2}(\e_c-\e_d)$ if $b\notin\{c,d\}$, and between $D\u=\frac{1}{\sqrt2}(\e_a+\e_b)$ and $D\v=\pm\frac{1}{\sqrt2}(\e_c+\e_d)$ otherwise. The same argument works for 2 by taking $D$ such that $D_a=D_{c}=-D_b=-D_{d}=1$.
\end{proof}

\begin{rem}
\label{rem3}
If $G$ is a positively weighted graph with perfect state transfer between $\frac{1}{\sqrt2}(\e_a+\e_b)$ and $\frac{1}{\sqrt2}(\e_c+\e_d)$, then the conclusion in Corollary \ref{cor} holds whenever we replace every pair state with a plus state, and every plus state with a pair state.
\end{rem}

Applying a sign change to the edges of a graph with edge-perturbed twin subgraphs yields the next result. Note that $U_{\tilde{G}_D}(t)$ denotes the transition matrix of $\tilde{G}_D$. 

\begin{thm}\label{3t3}
Let $G$ be a bounded graph with edge-perturbed twin subgraphs $X_1$ and $X_2$ with isomorphism $f:X_1\rightarrow X_2$. Suppose $D$ is a diagonal matrix of $\pm 1$'s 
and let $\tilde{G}_D$ be the graph satisfying $DA(G)D=A(\tilde{G}_D)$. 
If $U'(t)=\exp{\ob{it\ob{A\ob{X_1}-A'}}}$, then
    \[
    (DQ)^T U_{\tilde{G}_D}(t)DQ=\begin{bmatrix}
     U'(t) & \mathbf{0}\\
     \mathbf{0} & U_{G/\Pi}(t)
     \end{bmatrix},\quad \text{for all $t\in\mathbb{R}$}.
    \]
\end{thm}

\begin{proof}
This follows from Theorem \ref{mainthm} and the fact that $U_{G}(t)=DU_{\tilde{G}_D}(t)D$. 
\end{proof}


\section{Pair and plus state transfer}\label{sec:prpl}

We now investigate pair and plus state transfer in a graph $\tilde{G}_D$ obtained from a graph $G$ with edge-perturbed twin subgraphs.

\begin{cor}
\label{mainres}
Let $G$ be a bounded graph with edge-perturbed twin subgraphs $X_1$ and $X_2$, and let $a,b$ be vertices of $X_1$. The following are equivalent.
\begin{enumerate}
\item Perfect state transfer occurs between $a$ and $b$ in $X_1$ relative to $A(X_1)-A'$.
\item Perfect state transfer occurs between $\frac{1}{\sqrt2}(\e_a-\e_{\hat{f}(a)})$ and $\frac{1}{\sqrt2}(\e_b-\e_{\hat{f}(b)})$ in $G$.%
\item Perfect state transfer occurs between $\frac{1}{\sqrt2}(\e_a+\e_{\hat{f}(a)})$ and $\frac{1}{\sqrt2}(\e_b+\e_{\hat{f}(b)})$ in $\tilde{G}_D$, where $D$ is a diagonal matrix of $\pm 1$'s with $D_a=-D_{\hat{f}(a)}$ and $D_b=-D_{\hat{f}(b)}$.
\item Perfect state transfer occurs between $\frac{1}{\sqrt2}(\e_a+\e_{\hat{f}(a)})$ and $\frac{1}{\sqrt2}(\e_b-\e_{\hat{f}(b)})$ in $\tilde{G}_D$, where $D$ is a diagonal matrix of $\pm 1$'s with $D_a=-D_{\hat{f}(a)}$ and $D_b=D_{\hat{f}(b)}$.
\end{enumerate}
Furthermore, perfect state transfer occurs at the same time in 1-4.
\end{cor}



\begin{proof}
The equivalence of 1 and 2 is immediate from Corollary \ref{maincor}(1). Let $\delta=\pm 1$. The equivalence of $2$ and $3$ follows from Theorem \ref{3t3} by taking $D$ such that $D_a=-D_{\hat{f}(a)}$ and $D_b=-D_{\hat{f}(b)}$, in which case $DQ\check{\e}_a=\frac{\delta}{\sqrt{2}}\ob{\e_a+\e_{\hat{f}(a)}}$ and $DQ\check{\e}_b=\frac{\delta}{\sqrt{2}}\ob{\e_b+\e_{\hat{f}(b)}}$. The equivalence of 2 and 4 follows from Theorem \ref{3t3} by taking $D$ such that $D_a=-D_{\hat{f}(a)}$ and $D_b=D_{\hat{f}(b)}$, in which case $DQ\check{\e}_a=\frac{\delta}{\sqrt{2}}\ob{\e_a+\e_{\hat{f}(a)}}$ and $DQ\check{\e}_b=\frac{\delta}{\sqrt{2}}\ob{\e_b-\e_{\hat{f}(b)}}$.
\end{proof}

The equivalence of 1 and 2 in Corollary \ref{mainres} was used by Pal and Mohapatra to obtain an infinite family of trees with maximum degree three admitting pair PST \cite{pal10}.
    
\begin{cor}
\label{mainres1}
Let $G$ be a bounded graph and $\Pi$ be an equitable partition of $G$ with cells $V_1=\{a,b\}$ and $V_2=\{c,d\}$.
The following are equivalent.
\begin{enumerate}
\item Perfect state transfer occurs between $V_1$ and $V_2$ in $G/\Pi$.
\item Perfect state transfer occurs between $\frac{1}{\sqrt2}(\e_a+\e_b)$ and $\frac{1}{\sqrt2}(\e_c+\e_d)$ in $G$.
\item Perfect state transfer occurs between $\frac{1}{\sqrt2}(\e_a-\e_{b})$ and $\frac{1}{\sqrt2}(\e_c-\e_{d})$ in $\tilde{G}_D$, where $D$ is a diagonal matrix of $\pm 1$'s with $D_a=-D_{b}$ and $D_c=-D_{d}$.
\item Perfect state transfer occurs between $\frac{1}{\sqrt2}(\e_a-\e_{b})$ and $\frac{1}{\sqrt2}(\e_c+\e_{d})$ in $\tilde{G}_D$, where $D$ is a diagonal matrix of $\pm 1$'s with $D_a=-D_b$ and $D_c=D_d$.
\end{enumerate}
Furthermore, perfect state transfer occurs at the same time in 1-4.
\end{cor}

\begin{proof}
The equivalence of 1 and 2 follows from Corollary \ref{maincor}(3) by taking $(\u,\v)=(\e_1,\e_2)$ and
$(\mathcal{C}\e_1,\mathcal{C}\e_2)=\left(\frac{1}{\sqrt2}(\e_a+\e_b),\frac{1}{\sqrt2}(\e_c+\e_d)\right)$. The rest follows using the same argument in the proof of Corollary \ref{mainres}.
\end{proof}



\begin{exa}
Let $A(G)$ be the adjacency matrix of the unweighted graph $G$ given in Figure \ref{3fig1}(a). Let $V_1=\cb{1,2},$ $V_2=\cb{3},$ $V_3=\cb{4,5},$ and $V_4=\cb{6}$ be disjoint cells of an equitable partition $\Pi$ of the vertex set of $G.$ The characteristic matrix $\mathcal{C}$ and the adjacency matrix $A(G/\Pi)$ of the symmetrized quotient graph of $G$ over $\Pi$ are as follows.

\[\mathcal{C}=\begin{bmatrix}
\frac{1}{\sqrt{2}} & \frac{1}{\sqrt{2}} & 0 & 0 & 0 & 0\\
0 & 0 & 1 & 0 & 0 & 0\\
0 & 0 & 0 & \frac{1}{\sqrt{2}} & \frac{1}{\sqrt{2}} & 0\\
0 & 0 & 0 & 0 & 0 & 1
\end{bmatrix}^T, \quad A(G/\Pi)=\begin{bmatrix}
0 & \sqrt{2} & 0 & \sqrt{2}\\
\sqrt{2} & 0 & \sqrt{2} & 0 \\
0 & \sqrt{2} & 0 & \sqrt{2} \\
\sqrt{2} & 0 & \sqrt{2} & 0
\end{bmatrix}.\]
Let $D'=\text{diag}\ob{1, -1, 1, 1, -1, 1}$, $D''=\text{diag}\ob{1, 1, 1, 1, -1, 1}.$ We observe the following.

\begin{enumerate}
\item The adjacency matrix of a weighted $C_4$ with each edge having a weight $\sqrt{2}$ is $A(G/\Pi)$. Since the unweighted $C_4$ has perfect state transfer between vertices $1$ and $3$ at $\frac{\pi}{2}$, $U_{G/\Pi}\ob{\frac{\pi}{2\sqrt 2}}\check{\e}_1=-\check{\e}_3,$ where $\check{\e}_1=\tb{1~0~0~0}^T$ and $\check{\e}_3=\tb{0~0~1~0}^T.$ By Corollary \ref{mainres1}(2), we get perfect state transfer between $\frac{1}{\sqrt2}(\e_1+\e_2)$ and $\frac{1}{\sqrt2}(\e_4+\e_5)$ in $G$ at $\frac{\pi}{2\sqrt 2}$. 

\item Using $D'$, Corollary \ref{mainres1}(3) yields perfect state transfer between $\frac{1}{\sqrt2}(\e_1-\e_2)$ and $\frac{1}{\sqrt2}(\e_4-\e_5)$ at $\frac{\pi}{2\sqrt 2}$ in $\tilde{G}_{D'}$, see Figure \ref{3fig1}(b). Using $D''$, Corollary \ref{mainres1}(4) yields perfect state transfer at $\frac{\pi}{2\sqrt 2}$ between $\frac{1}{\sqrt2}(\e_1+\e_2)$ and $\frac{1}{\sqrt2}(\e_4-\e_5)$ in $\tilde{G}_{D''}$, see Figure \ref{3fig1}(c). 
\end{enumerate}
\end{exa}

\begin{rem}
Corollaries \ref{mainres} and \ref{mainres1} hold when we replace `perfect state transfer' with `pretty good state transfer'. In this case, PGST occurs at the same sequence of times in 1-4.
\end{rem}

\begin{figure}
\begin{multicols}{3}
\begin{center}
                    \begin{tikzpicture}[scale=0.6,auto=center]
                       \tikzstyle{every node}=[circle, thick, black!90, fill=white, scale=0.65]
              \node[draw,minimum size=0.55cm, 
                inner sep=0 pt] (1) at (-3,3) {$1$};
				  \node[draw,minimum size=0.55cm, inner sep=0 pt] (2) at (3,3) {$2$};
				  \node[draw,minimum size=0.55cm, 
                  inner sep=0 pt] (3) at (0,0.2)  {$3$};
                  \node[draw,minimum size=0.55cm, 
                  inner sep=0 pt] (4) at (-1.5,3) {$4$};
                  \node[draw,minimum size=0.55cm, inner sep=0 pt] (5) at (1.5, 3) {$5$};
                  \node[draw,minimum size=0.55cm, inner sep=0 pt] (6) at (0,5.8) {$6$};

            \draw[thick, black!90] (1)-- (3)--(2)--(6)--(1);
            \draw[thick, black!90](2)--(3)--(4)--(6)--(5)--(3);
          
                \end{tikzpicture}

                  \begin{tikzpicture}[scale=0.6,auto=center]
                       \tikzstyle{every node}=[circle, thick, black!90, fill=white, scale=0.65]
              \node[draw,minimum size=0.55cm, 
                inner sep=0 pt] (1) at (-3,3) {$1$};
				  \node[draw,minimum size=0.55cm, inner sep=0 pt] (2) at (3,3) {$2$};
				  \node[draw,minimum size=0.55cm, 
                  inner sep=0 pt] (3) at (0,0.2)  {$3$};
                  \node[draw,minimum size=0.55cm, 
                  inner sep=0 pt] (4) at (-1.5,3) {$4$};
                  \node[draw,minimum size=0.55cm, inner sep=0 pt] (5) at (1.5, 3) {$5$};
                  \node[draw,minimum size=0.55cm, inner sep=0 pt] (6) at (0,5.8) {$6$};

            \draw[thick, black!90] (1)-- (3);
            \draw[thick, black!90](3)--(4)--(6)--(1);
            \draw[thick, dashed, black!90](3)--(2)--(6)--(5)--(3);

\end{tikzpicture}

                \begin{tikzpicture}[scale=0.6,auto=center]
                       \tikzstyle{every node}=[circle, thick, black!90, fill=white, scale=0.65]
              \node[draw,minimum size=0.55cm, 
                inner sep=0 pt] (1) at (-3,3) {$1$};
				  \node[draw,minimum size=0.55cm, inner sep=0 pt] (2) at (3,3) {$2$};
				  \node[draw,minimum size=0.55cm, 
                  inner sep=0 pt] (3) at (0,0.2)  {$3$};
                  \node[draw,minimum size=0.55cm, 
                  inner sep=0 pt] (4) at (-1.5,3) {$4$};
                  \node[draw,minimum size=0.55cm, inner sep=0 pt] (5) at (1.5, 3) {$5$};
                  \node[draw,minimum size=0.55cm, inner sep=0 pt] (6) at (0,5.8) {$6$};

            \draw[thick, black!90] (1)-- (3)--(2)--(6);
            \draw[thick, black!90](3)--(4)--(6)--(1);
            \draw[thick, dashed, black!90](3)--(5)--(6);

                \end{tikzpicture}
                \end{center}
\end{multicols}
		\caption{Perfect state transfer between $\frac{1}{\sqrt2}(\e_1+\e_2)$ and $\frac{1}{\sqrt2}(\e_4+\e_5)$ (left), between $\frac{1}{\sqrt2}(\e_1-\e_2)$ and $\frac{1}{\sqrt2}(\e_4-\e_5)$ (centre), and between $\frac{1}{\sqrt2}(\e_1+\e_2)$ and $\frac{1}{\sqrt2}(\e_4-\e_5)$ (right)}
  \label{3fig1}
	\end{figure}
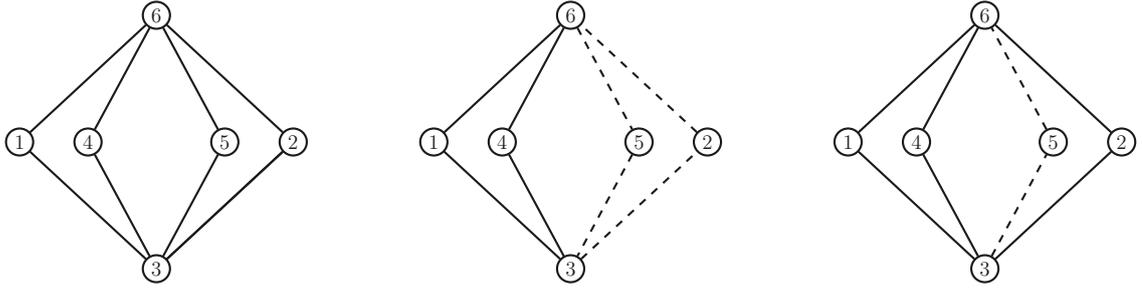 

\section{Infinite families}\label{sec:inf}

In this section, we assume that the graphs and signed graphs are unweighted (i.e., they have edge weights 1 and $\pm 1$, respectively).

It is known that $P_2$ admits perfect state transfer between the vertices $1$ and $2$ at $\frac{\pi}{2}$. Taking the graph on the left of Figure \ref{fi}, 
we find that for any graph $H$ identified at vertex 3, perfect state transfer occurs in $G$ between $\frac{1}{\sqrt2}(\e_1-\e_5)$ and $\frac{1}{\sqrt2}(\e_2-\e_4)$ at $\frac{\pi}{2}$ by Corollary \ref{mainres}(2). 
In particular, if $H$ is a graph, planar graph or tree with maximum degree $k$, then:

\begin{thm}
\label{inftreesmaxdeg3}
For each $k\geq 3$, there are infinitely many connected unweighted graphs (resp., planar graphs, trees) with maximum degree $k$ admitting pair state transfer at $\frac{\pi}{2}$.
\end{thm}

Theorem \ref{inftreesmaxdeg3} complements the fact that for every positive integer $k$, there are only finitely many connected graphs with maximum degree $k$ admitting plus state transfer \cite[Corollary 3.5]{kim}. If $k=3$, then Theorem \ref{inftreesmaxdeg3} recovers a result of Pal and Mohapatra \cite{pal10}.

\begin{cor}
\label{inftreesmaxdeg31}
There are infinitely many unweighted trees with maximum degree three admitting pair state transfer at $\frac{\pi}{2}$.
\end{cor}

\begin{figure}
\begin{multicols}{2}
\begin{center}
\begin{tikzpicture}[scale=.5,auto=left]
                       \tikzstyle{every node}=[circle, thick, fill=white, scale=0.6]
                       
		        \node[draw] (1) at (1.5,0) {$3$};		        
		        \node[draw,minimum size=0.7cm, inner sep=0 pt] (2) at (-2.8, 1) {$1$};
		        \node[draw,minimum size=0.7cm, inner sep=0 pt] (3) at (-1.3, 2) {$2$};
		        \node[draw,minimum size=0.7cm, inner sep=0 pt] (4) at (-2.8, -1) {$5$};
		        \node[draw,minimum size=0.7cm, inner sep=0 pt] (5) at (-1.3, -2) {$4$};		       
	
				\node at (-4.5, 2) {$X_1$};
				\node at (-4.5,-2) {$X_2$};
				\node at (6.3,0.5) {$H$};
				
				\draw[dotted] (-2,1.5) ellipse (2 cm and 1.3 cm);
				\draw[dotted] (-2,-1.5) ellipse (2 cm and 1.3 cm);
				\draw[dotted] (3.2,0.4) circle (2.5 cm);
								
				\draw [thick, black!70] (1)--(3)--(2);
				\draw [thick, black!70] (5)--(4);
                \draw [thick, dashed, black!70] (1)--(5);

				\draw[thick, black!70] (1)..controls (3,4) and (8,0)..(1);

				\end{tikzpicture}
                
\begin{tikzpicture}[scale=.5,auto=left]
                       \tikzstyle{every node}=[circle, thick, fill=white, scale=0.6]
                       
		        \node[draw] (1) at (1.5,0) {$3$};		        
		        \node[draw,minimum size=0.7cm, inner sep=0 pt] (2) at (-2.8, 1) {$1$};
		        \node[draw,minimum size=0.7cm, inner sep=0 pt] (3) at (-1.3, 2) {$2$};
		        \node[draw,minimum size=0.7cm, inner sep=0 pt] (4) at (-2.8, -1) {$5$};
		        \node[draw,minimum size=0.7cm, inner sep=0 pt] (5) at (-1.3, -2) {$4$};		       
	
				\node at (-4.5, 2) {$X_1$};
				\node at (-4.5,-2) {$X_2$};
				\node at (6.3,0.5) {$H$};
				
				\draw[dotted] (-2,1.5) ellipse (2 cm and 1.3 cm);
				\draw[dotted] (-2,-1.5) ellipse (2 cm and 1.3 cm);
				\draw[dotted] (3.2,0.4) circle (2.5 cm);
								
				\draw [thick, black!70] (1)--(3)--(2);
				\draw [thick, dashed, black!70] (5)--(4);
                \draw [thick, black!70] (1)--(5);

				\draw[thick, black!70] (1)..controls (3,4) and (8,0)..(1);

				\end{tikzpicture}
				
\end{center}
\end{multicols}
\caption{\label{fig1} Signed versions of a graph with $P_2$ as twin subgraphs}
\end{figure}
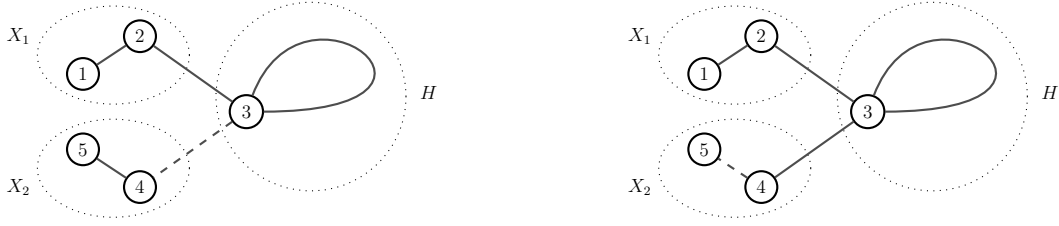

Now, if we assign the edge $\{3,4\}$ in the graph on the left of Figure \ref{fi} a negative weight, we obtain the graph on the left of Figure \ref{fig1}, which has perfect state transfer at $\frac{\pi}{2}$ between $\frac{1}{\sqrt2}(\e_1+\e_5)$ and $\frac{1}{\sqrt2}(\e_2+\e_4)$ by Corollary \ref{mainres}(3). This yields the following results analogous to Theorem \ref{inftreesmaxdeg3} and Corollary \ref{inftreesmaxdeg31}, respectively. 

\begin{thm}
\label{infsignedplus}
For each $k\geq 3$, there are infinitely many connected unweighted signed graphs $\tilde{G}$ such that the underlying graphs have maximum degree $k$, $\tilde{G}$ has exactly one negative edge weight, and plus state transfer in $\tilde{G}$ at $\frac{\pi}{2}$.
\end{thm}

\begin{cor}
\label{treemaxdeg3}
There are infinitely many unweighted signed trees $\tilde{T}$ such that the underlying graphs have maximum degree three, $\tilde{T}$ has exactly one negative edge weight and plus state transfer in $\tilde{T}$ at $\frac{\pi}{2}$.
\end{cor}

Applying Corollary \ref{mainres}(4) and Theorem \ref{inftreesmaxdeg3} to the graph on the right Figure \ref{fig1} yields the following result.

\begin{thm}
\label{ak}
For each $k\geq 3$, there are infinitely many connected unweighted signed graphs $\tilde{G}$ such that the underlying graphs have maximum degree $k$ and perfect state transfer occurs between a plus state and a pair state in $\tilde{G}$ at $\frac{\pi}{2}$.
\end{thm}

\begin{cor}\label{3cor9}
There are infinitely many unweighted signed trees $\tilde{T}$ such that the underlying trees have maximum degree three, $\tilde{T}$ has exactly one negative edge weight and perfect state transfer occurs between a plus state and a pair state in $\tilde{T}$ at $\frac{\pi}{2}$.
\end{cor}

\begin{rem}
Since $H$ in both figures in Figure \ref{fig1} may be taken to be finite or infinite, Theorems \ref{inftreesmaxdeg3}-\ref{ak} and Corollaries \ref{inftreesmaxdeg31}-\ref{3cor9} hold for finite graphs and bounded infinite graphs.
\end{rem}

For the class of finite graphs, we obtain the following results.

\begin{thm}
\label{almostallplanar}
Almost all labelled connected unweighted planar finite graphs admit pair state transfer at $\frac{\pi}{2}$. Moreover, almost all labelled connected unweighted planar finite graphs can be assigned a single negative edge weight so that the resulting signed graphs have plus state transfer, or perfect state transfer between a pair state and a plus state, both at $\frac{\pi}{2}$.
\end{thm}

\begin{proof}
Let $\mathcal{CP}_n$ denote the class of labelled connected unweighted planar graphs on $n$ vertices. Adapting the proof of Theorem 4.3 in \cite{ciardo2020braess} by replacing $H$ as $P_5$ rooted at the vertex adjacent to two degree two vertices (in lieu of a $P_3$ rooted at a central vertex), we conclude that a labelled connected planar graph on $n$ vertices (sampled from $\mathcal{CP}_n$ uniformly at random) contains $P_5$ as a rooted branch – actually linearly many – with probability approaching 1 as $n$ approaches infinity. Therefore, almost all labelled connected planar graphs have the same form as the graph on the left of Figure \ref{fi}, where $H$ is connected planar graph. Combining this with our observations at the start of the section and Corollary \ref{mainres}(2-4) yields the desired result.
\end{proof}

An analogous result also holds for trees. 

\begin{thm}
\label{almostalltrees}
Almost all unweighted finite trees admit pair state transfer at $\frac{\pi}{2}$. Moreover, almost all unweighted trees can be assigned a single negative edge weight so that the resulting signed trees have plus state transfer, or perfect state transfer between a pair state and a plus state, both at $\frac{\pi}{2}$.
\end{thm}

\begin{proof}
In \cite[Theorem 7]{schwenk1973almost}, Schwenk showed that for any rooted tree $L$, almost all trees contain $L$ as a limb. Taking $L$ as $P_5$ rooted at the vertex adjacent to two degree two vertices, we get that almost all trees have the same form as the graph on the left of Figure \ref{fi}, where $H$ is a tree. Combining this with our observations at the start of the section and Corollary \ref{mainres}(2-4) yields the result.   
\end{proof}

\begin{rem}
\label{rem}
We may also consider graphs with twin subgraphs isomorphic to $P_3$ (in lieu of $P_2$) in Figure \ref{fi}. Since $P_3$ has PST between end vertices at $\frac{\pi}{\sqrt{2}}$. Applying the same argument to the graphs in Figure \ref{figg}, Theorems \ref{inftreesmaxdeg3}-\ref{almostalltrees} and Corollaries \ref{inftreesmaxdeg31}-\ref{3cor9} also hold at $\frac{\pi}{\sqrt{2}}$.
\end{rem}

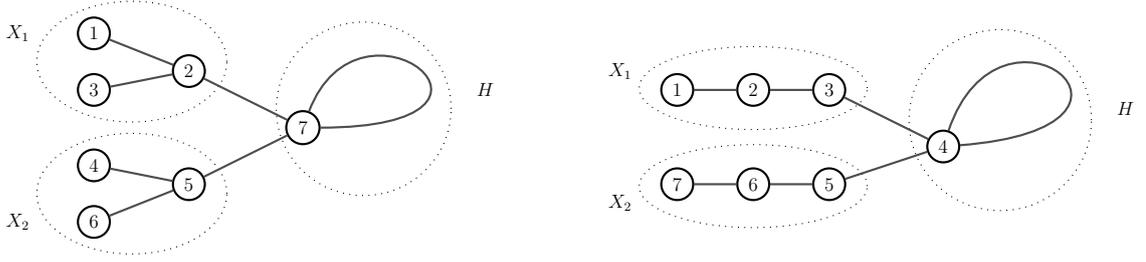
\begin{figure}
\begin{multicols}{2}
\begin{center}
\begin{tikzpicture}[scale=.5,auto=left]
                       \tikzstyle{every node}=[circle, thick, fill=white, scale=0.6]
                        \node[draw,minimum size=0.7cm, inner sep=0 pt] (1) at (-4, 2.5) {$1$};
          \node[draw,minimum size=0.7cm, inner sep=0 pt] (2) at (-1.5, 1.5) {$2$};
          \node[draw,minimum size=0.7cm, inner sep=0 pt] (3) at (-4, 1) {$3$};
          \node[draw,minimum size=0.7cm, inner sep=0 pt] (4) at (-4, -1) {$4$};
          \node[draw,minimum size=0.7cm, inner sep=0 pt] (5) at (-1.5, -1.5) {$5$};
          \node[draw,minimum size=0.7cm, inner sep=0 pt] (6) at (-4, -2.5) {$6$}; 
          \node[draw] (7) at (1.5,0) {$7$};
	\node at (-6, 2.5) {$X_1$};
				\node at (-6,-2.5) {$X_2$};
				\node at (6.3,1) {$H$};
				\draw[dotted] (-3,1.75) ellipse (2.5 cm and 1.6cm);
				\draw[dotted] (-3,-1.75) ellipse (2.5 cm and 1.6cm);
				\draw[dotted] (3.1,0.5) circle (2.3 cm);
				\draw [thick, black!70] (1)--(2)--(3);
                \draw [thick, black!70] (2)--(7)--(5)--(4);
				\draw [thick, black!70] (5)--(6);
                \draw[thick, black!70] (7)..controls (3,4) and (8,0)..(7);
\end{tikzpicture}
            \begin{tikzpicture}[scale=.5,auto=left]
                       \tikzstyle{every node}=[circle, thick, fill=white, scale=0.6]
                       \node[draw,minimum size=0.7cm, inner sep=0 pt] (1) at (-5.5, 1) {$1$};
                \node[draw,minimum size=0.7cm, inner sep=0 pt] (2) at (-3.5, 1) {$2$};
		        \node[draw,minimum size=0.7cm, inner sep=0 pt] (3) at (-1.5, 1) {$3$};
                \node[draw,minimum size=0.7cm, inner sep=0 pt] (4) at (1.5, -0.5) {$4$};	
		        \node[draw,minimum size=0.7cm, inner sep=0 pt] (5) at (-1.5, -1.5) {$5$};
		     \node[draw,minimum size=0.7cm, inner sep=0 pt] (6) at (-3.5, -1.5) {$6$};   	        
	\node[draw,minimum size=0.7cm, inner sep=0 pt] (7) at (-5.5, -1.5) {$7$};

                \node at (-7, 1.5) {$X_1$};
				\node at (-7,-2) {$X_2$};
				\node at (6.3,0.5) {$H$};
                \draw [dotted] (-3.5,1.05) ellipse (3cm and 1.1cm);
                \draw [dotted] (-3.5,-1.55) ellipse (3cm and 1.1cm);
				\draw[dotted] (3,0.2) circle (2.4cm);
						\draw [thick, black!70] (1)--(2)--(3)--(4)--(5)--(6)--(7);
				\draw[thick, black!70] (4)..controls (3,4) and (8,0)..(4);
\end{tikzpicture}
			\end{center}
\end{multicols}
\caption{\label{figg} Graphs having $P_3$ as twin subgraphs}
\end{figure}

\begin{cor}
\label{more}
For each $k\geq 2$, there are more connected unweighted finite graphs (resp., planar graphs, trees) with maximum degree $k$ admitting pair state transfer than those admitting vertex state transfer. This holds for all $k\geq 3$ if we replace `finite' by `infinite'.
\end{cor}

\begin{proof}
For finite and infinite graphs (resp., planar graphs, trees), the case $k\geq 3$ follows from Theorem \ref{inftreesmaxdeg3}, \cite[Corollary 6.2]{god2} and the fact that vertex PST does not occur in infinite graphs \cite[Theorem 5.6]{godd}. For finite graphs, the case $k=2$ follows from the fact that there are only three such graphs with maximum degree two admitting vertex PST ($K_2$, $P_3$ and $C_4$), while there are six admitting pair state transfer ($P_3$, $P_5$, $P_7$, $C_4$, $C_6$, and $C_8$) (see \cite[Corollary 7.4]{godsil2025perfect} and \cite[Theorem 6.5]{kim}). 
The case $k=2$ for finite planar graphs (resp., trees) follows similarly.
\end{proof}


\section{Blow-ups}\label{sec:blow}

For every positive integer $n,$ consider $\mathbb{Z}_n=\cb{0, 1, 2, \ldots, n-1}.$ The blow-up of $n$ copies of $H$, denoted by $\overset{n}{\uplus}~H$, is the graph with vertex set $\mathbb{Z}_n\times V(H),$ and adjacency matrix
\begin{center}
$J_n\otimes A(H).$
\end{center}
Using the blow-up operation, we generate new examples of pair and plus PST.
    
\begin{thm}
\label{corbu}
Let $a$ and $b$ be vertices in $H$, $X_n:=\overset{n}{\uplus}~H$, and $D$ be a diagonal matrix of $\pm 1$'s. The following are equivalent.
\begin{enumerate}
\item Perfect state transfer occurs between vertices $a$ and $b$ in $H$.
\item Perfect state transfer occurs between $\sum_{j=0}^{n-1}\e_{(j,a)}$ and $\sum_{j=0}^{n-1}\e_{(j,b)}$ in $X_n$. In particular, perfect state transfer occurs between $\frac{1}{\sqrt2}(\e_{(0,a)}+\e_{(1,a)})$ and $\frac{1}{\sqrt2}(\e_{(0,b)}+\e_{(1,b)})$ in $X_2$.
\item The signed graph $\tilde{X_2}$ with $A(\tilde{X_2})=DA(X_2)D$ where $D_{(0,a)}=-D_{(1,a)}$ and $D_{(0,b)}=-D_{(1,b)}$ has perfect state transfer between $\frac{1}{\sqrt2}(\e_{(0,a)}-\e_{(1,a)})$ and $\frac{1}{\sqrt2}(\e_{(0,b)}-\e_{(1,b)})$.
\item The signed graph $\tilde{X_2}$ with $A(\tilde{X_2})=DA(X_2)D$ where $D_{(0,a)}=-D_{(1,a)}$ and $D_{(0,b)}=D_{(1,b)}$ has perfect state transfer between $\frac{1}{\sqrt2}(\e_{(0,a)}-\e_{(1,a)})$ and $\frac{1}{\sqrt2}(\e_{(0,b)}+\e_{(1,b)})$.
\end{enumerate}
Further, perfect state transfer occurs in 1 at $\tau$, in 2 at $\frac{\tau}{n}$ and in 3-4 at $\frac{\tau}{2}$.
\end{thm}

\begin{proof}
For each $a\in V(H)$, $\cb{(j,a):j\in\Zl_n}$ forms a cell of equitable partition $\Pi$ of $V(X_n)$. Applying Corollary \ref{maincor} and noting that $A(X_n/\Pi)=nA(H)$ yields the equivalence of 1 and the first statement in 2. Meanwhile, applying Corollary \ref{mainres1}(3,4) with $G=X_2$ yields the equivalence of 3, 4 and the second statement in 2.
\end{proof}

\begin{rem}
Since $(0,a)$ and $(1,a)$ are twins in $\overset{2}{\uplus}~H$, the pair state $\frac{1}{\sqrt2}(\e_{(0,a)}-\e_{(1,a)})$ is an eigenvector for $A(\overset{2}{\uplus}~H)$, so it is not involved in perfect state transfer in $\overset{2}{\uplus}~H$ \cite{mon2}.
\end{rem}

\begin{rem}
\label{pgstbu}
Theorem \ref{corbu} holds when we replace `perfect state transfer' with `pretty good state transfer'. In this case, PGST occurs in 1 relative to $\{\tau_k\}$, in 2 relative to $\{\frac{\tau_k}{n}\}$, and in 3-4 relative to $\{\frac{\tau_k}{2}\}$.
\end{rem}

\begin{exa}
It is known that $P_3$ admits PST between vertices $a$ and $c$ at $\frac{\pi}{\sqrt{2}}$. Applying Theorem \ref{corbu}, we get PST between $\frac{1}{\sqrt2}(\e_{(0,a)}+\e_{(1,a)})$ and $\frac{1}{\sqrt2}(\e_{(0,c)}+\e_{(1,c)})$ in $\overset{2}{\uplus}~P_3$, between $\frac{1}{\sqrt2}(\e_{(0,a)}-\e_{(1,a)})$ and $\frac{1}{\sqrt2}(\e_{(0,c)}+\e_{(1,c)})$ in a signed $\overset{2}{\uplus}~P_3$, and between $\frac{1}{\sqrt2}(\e_{(0,a)}-\e_{(1,a)})$ and $\frac{1}{\sqrt2}(\e_{(0,c)}-\e_{(1,c)})$ in a signed $\overset{2}{\uplus}~P_3$ all at $\frac{\pi}{2\sqrt{2}}$ (see Figure \ref{3fig3}).
\end{exa}

\begin{figure}
\begin{multicols}{3}
\begin{center}
                    \begin{tikzpicture}[scale=1,auto=center]
                       \tikzstyle{every node}=[circle, thick, black!90, fill=white, scale=0.65]
               \node[draw,minimum size=0.55cm, 
                  inner sep=0 pt] (1) at (-2,0)  {$(0,a)$};     
              \node[draw,minimum size=0.55cm, 
                inner sep=0 pt] (2) at (0,0) {$(0,b)$};
				  \node[draw,minimum size=0.55cm, inner sep=0 pt] (3) at (2,0) {$(0,c)$};
                   \node[draw,minimum size=0.55cm, inner sep=0 pt] (4) at (-2,1.4) {$(1,a)$};
                  \node[draw,minimum size=0.55cm, inner sep=0 pt] (5) at (0,1.4) {$(1,b)$};
                 \node[draw,minimum size=0.55cm, 
                  inner sep=0 pt] (6) at (2,1.4) {$(1,c)$};

           \draw[thick, black!90] (1)-- (2)--(3);
            \draw[thick, black!90](4)--(5)--(6);
            \draw[thick, black!90] (1)-- (5)--(3);
              \draw[thick, black!90](6)--(2)--(4);
                \end{tikzpicture}

                 \begin{tikzpicture}[scale=1,auto=center]
                       \tikzstyle{every node}=[circle, thick, black!90, fill=white, scale=0.65]
               \node[draw,minimum size=0.55cm, 
                  inner sep=0 pt] (1) at (-2,0)  {$(0,a)$};     
              \node[draw,minimum size=0.55cm, 
                inner sep=0 pt] (2) at (0,0) {$(0,b)$};
				  \node[draw,minimum size=0.55cm, inner sep=0 pt] (3) at (2,0) {$(0,c)$};
                   \node[draw,minimum size=0.55cm, inner sep=0 pt] (4) at (-2,1.4) {$(1,a)$};
                  \node[draw,minimum size=0.55cm, inner sep=0 pt] (5) at (0,1.4) {$(1,b)$};
                 \node[draw,minimum size=0.55cm, 
                  inner sep=0 pt] (6) at (2,1.4) {$(1,c)$};
                 
            \draw[thick, black!90] (1)-- (2)--(3);
            \draw[thick, black!90](5)--(6);
            \draw[thick, black!90] (1)-- (5)--(3);
           
              \draw[thick, black!90] (6)--(2);
        \draw[thick, dashed, black!90](4)--(5);
        \draw[thick, dashed, black!90] (2)--(4);
		
                \end{tikzpicture}

                \begin{tikzpicture}[scale=1,auto=center]
                       \tikzstyle{every node}=[circle, thick, black!90, fill=white, scale=0.65]
               \node[draw,minimum size=0.55cm, 
                  inner sep=0 pt] (1) at (-2,0)  {$(0,a)$};     
              \node[draw,minimum size=0.55cm, 
                inner sep=0 pt] (2) at (0,0) {$(0,b)$};
				  \node[draw,minimum size=0.55cm, inner sep=0 pt] (3) at (2,0) {$(0,c)$};
                   \node[draw,minimum size=0.55cm, inner sep=0 pt] (4) at (-2,1.4) {$(1,a)$};
                  \node[draw,minimum size=0.55cm, inner sep=0 pt] (5) at (0,1.4) {$(1,b)$};
                 \node[draw,minimum size=0.55cm, 
                  inner sep=0 pt] (6) at (2,1.4) {$(1,c)$};
                 
            \draw[thick, black!90](1)--(2)--(3);
            \draw[thick, black!90](1)--(5);
            \draw[thick, black!90] (5)--(3);
        \draw[thick, dashed, black!90](4)--(5)--(6);
        \draw[thick, dashed, black!90](6)--(2)--(4);
		
                \end{tikzpicture}
        
                \end{center}
\end{multicols}
		\caption{$\overset{2}{\uplus}~P_3$ with PST between $\frac{1}{\sqrt2}(\e_{(0,a)}+\e_{(1,a)})$ and $\frac{1}{\sqrt2}(\e_{(0,c)}+\e_{(1,c)})$ (left), signed $\overset{2}{\uplus}~P_3$ with PST between $\frac{1}{\sqrt2}(\e_{(0,a)}-\e_{(1,a)})$ and $\frac{1}{\sqrt2}(\e_{(0,c)}+\e_{(1,c)})$ (centre), signed $\overset{2}{\uplus}~P_3$ with PST between $\frac{1}{\sqrt2}(\e_{(0,a)}-\e_{(1,a)})$ and $\frac{1}{\sqrt2}(\e_{(0,c)}-\e_{(1,c)})$ (right)
        }
  \label{3fig3}
	\end{figure}
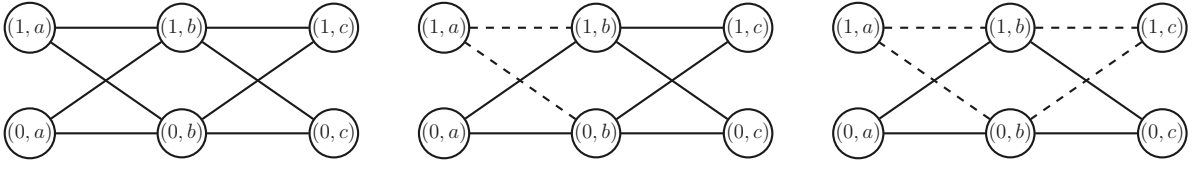

We now provide infinite families illustrating Theorem \ref{corbu}(1-2).

\begin{exa}
In \cite[Theorem 6]{pal9}, $\overset{2}{\uplus}~K_n$ has PST at $\frac{\pi}{2}$ between every pair of twin vertices $a$ and $b$ if and only if $n$ is even. Invoking Theorem \ref{corbu}(2), $\overset{2}{\uplus}\ob{\overset{2}{\uplus}~K_n}=\overset{4}{\uplus}~K_n$ admits plus state transfer at $\frac{\pi}{4}$ between $\frac{1}{\sqrt2}(\e_{(0,a)}+\e_{(1,a)})$ and $\frac{1}{\sqrt2}(\e_{(0,b)}+\e_{(1,b)})$ whenever $n$ is even.
\end{exa}

\begin{exa}
Let $X$ be a distance-regular graph that admits PST at $\tau$. By \cite[Corollary 4.5]{cou2}, $X$ is antipodal with classes of size two, and so PST occurs between any pair of vertices $a$ and $b$ at distance $d$. Invoking Theorem \ref{corbu}(2), $\overset{2}{\uplus}\ X$ admits plus state transfer at the same time between $\frac{1}{\sqrt2}(\e_{(0,a)}+\e_{(1,a)})$ and $\frac{1}{\sqrt2}(\e_{(0,b)}+\e_{(1,b)})$. In particular, this holds if we take $X=Q_d$, the hypercube of dimension $d$, for all $d\geq 1$ with $\tau=\frac{\pi}{2}$. 
\end{exa}

\begin{exa}
There is PGST in $P_n$ between vertices $a$ and $b$ if and only if $a+b=n+1$ and either (i) $n=2^t-1$, (ii) $n=p-1$ for some odd prime $p$, or (iii) $n=2^tp-1$ for some integer $t>0$ and odd prime $p$, and $a$ is a multiple of $2^{t-1}$ \cite[Theorem 4.3.3]{Bommel2019}. Invoking Theorem \ref{corbu}(2) and Remark \ref{pgstbu}, we get PGST in $\overset{2}{\uplus}\ P_n$ between $\frac{1}{\sqrt2}(\e_{(0,a)}+\e_{(1,a)})$ and $\frac{1}{\sqrt2}(\e_{(0,b)}+\e_{(1,b)})$ if and only if $a+b=n+1$ and one of conditions (i)-(iii) hold.
\end{exa}

\begin{exa}
There is PGST in $C_n$ between vertices $a$ and $b$ if and only if (i) $n$ is a power of two and (ii) $a$ and $b$ are antipodal vertices in $C_n$ \cite[Theorem 13]{pal4}. Again invoking Theorem \ref{corbu}(2) and Remark \ref{pgstbu}, we get PGST in $\overset{2}{\uplus}\ C_n$ between $\frac{1}{\sqrt2}(\e_{(0,a)}+\e_{(1,a)})$ and $\frac{1}{\sqrt2}(\e_{(0,b)}+\e_{(1,b)})$ if and only if conditions (i) and (ii) hold.
\end{exa}

\section{Sedentariness}\label{sec:sed}

Recall that a pure state $\u$ is \textit{$C$-sedentary} in $G$ if for some constant $0 < C \leq 1,$
 \[\inf_{t>0}~\mb{\u^*U_G(t)\u}\geq C.\]
If $C$ is not important, then we simply say that $\u$ is sedentary. In this section, we study the case when $\u$ is a vertex state, a plus state, or a pair state.
 
Vertex sedentariness was first studied in depth by Monterde in \cite{mon}, where she observed that a sedentary vertex in a graph does not have PGST. This observation extends to any pure state: a sedentary pure state is not involved in PGST. The most common example of a graph with sedentary vertices is the complete graph: for all $n\geq 3$, the vertex state $\e_a$ in $K_n$ is $C$-sedentary for all $a\in V(K_n)$, where $C=\frac{n-2}{n}$ \cite[Equation 6]{mon}.



The following result is analogous to Corollary \ref{mainres} and follows from Theorem \ref{3t3}.

\begin{cor}\label{3c5}
Let $G$ be a bounded graph with edge-perturbed twin subgraphs $X_1$ and $X_2$ and let $a$ be a vertex of $X_1$. 
The following are equivalent.
\begin{enumerate}
\item Vertex $a$ is $C$-sedentary in $X_1$ relative to $A(X_1)-A'$.
\item The pair state $\frac{1}{\sqrt2}(\e_a-\e_{\hat{f}(a)})$ is $C$-sedentary in $G$.
\item The plus state $\frac{1}{\sqrt2}(\e_a+\e_{\hat{f}(a)})$ is $C$-sedentary in $\tilde{G}_D$, where $D$ is a diagonal matrix of $\pm 1$'s with $D_a=-D_{\hat{f}(a)}$.
\end{enumerate}
\end{cor}

We say that vertices $u$ and $v$ are \textit{twins} in a weighted (signed or unsigned) graph $G$ if $N_G(u)\backslash\{v\}=N_G(v)\backslash\{u\}$ and $w(a,u)=w(a,v)$ for all $a\in N_G(u)\backslash\{v\}$. A maximal subset $T\subseteq V(G)$ is a \textit{twin set} in $G$ if every pair of vertices in $T$ are twins.

\begin{thm}
\label{twins}
Let $T$ be a twin set in $X_1$ with $|T|\geq 3$ and $C=1-\frac{2}{|T|}$. If $X_1$ and $X_2$ are twin subgraphs of a bounded graph $G$ and $a\in T$, then $\frac{1}{\sqrt2}(\e_a-\e_{\hat{f}(a)})$ is $C$-sedentary in $G$ and $\frac{1}{\sqrt2}(\e_a+\e_{\hat{f}(a)})$ is $C$-sedentary in $\tilde{G}_D$, where $D$ is given in Corollary \ref{3c5}(3).
\end{thm}

\begin{proof}
Since $T$ is a twin set in $X_1$ with $|T|\geq 3$, it follows that each $a\in T$ is $C$-sedentary in $X_1$ \cite[Theorem 16]{mon}. As $X_1$ and $X_2$ are twin subgraphs of $G$, it follows that $A'=0$, and so applying Corollary \ref{3c5} yields the desired result.
\end{proof}

\begin{figure}
\begin{multicols}{2}
\begin{center}
\begin{tikzpicture}[scale=.5,auto=left]
                       \tikzstyle{every node}=[circle, thick, fill=white, scale=0.6]
                       
		        \node[draw] (1) at (1.5,0.25) {$4$};		        
		        \node[draw,minimum size=0.7cm, inner sep=0 pt] (2) at (-3.5, 2.5) {$2$};
		        \node[draw,minimum size=0.7cm, inner sep=0 pt] (3) at (-1, 2.5) {$3$};
                    \node[draw,minimum size=0.7cm,
                    inner sep=0 pt] (6) at (-2.25, 1) {$1$};
		        \node[draw,minimum size=0.7cm, inner sep=0 pt] (4) at (-3.5, -2) {$6$};
		        \node[draw,minimum size=0.7cm, inner sep=0 pt] (5) at (-1, -2) {$7$};
                \node[draw,minimum size=0.7cm, inner sep=0 pt] (7) at (-2.25, -0.5) {$5$};

				\node at (6.3,0.5) {$H$};

				\draw[dotted] (3.2,0.5) circle (2.5 cm);
								
				\draw [thick, black!70] (1)--(3)--(2);
                \draw [thick, black!70] (2)--(6)--(3);
                \draw [thick, black!70] (5)--(7)--(4);
				\draw [thick, black!70] (5)--(4);
                \draw [thick, black!70] (1)--(5);

				\draw[thick, black!70] (1)..controls (3,4) and (8,0)..(1);
				\end{tikzpicture}

                \begin{tikzpicture}[scale=.5,auto=left]
                       \tikzstyle{every node}=[circle, thick, fill=white, scale=0.6]
                       
		        \node[draw] (1) at (1.5,0.25) {$4$};		        
		        \node[draw,minimum size=0.7cm, inner sep=0 pt] (2) at (-3.5, 2.5) {$2$};
		        \node[draw,minimum size=0.7cm, inner sep=0 pt] (3) at (-1, 2.5) {$3$};
                    \node[draw,minimum size=0.7cm,
                    inner sep=0 pt] (6) at (-2.25, 1) {$1$};
		        \node[draw,minimum size=0.7cm, inner sep=0 pt] (4) at (-3.5, -2) {$6$};
		        \node[draw,minimum size=0.7cm, inner sep=0 pt] (5) at (-1, -2) {$7$};
                \node[draw,minimum size=0.7cm, inner sep=0 pt] (7) at (-2.25, -0.5) {$5$};

				\node at (6.3,0.5) {$H$};

				\draw[dotted] (3.2,0.5) circle (2.5 cm);
								
				\draw [thick, black!70] (1)--(3)--(2);
                \draw [thick, black!70] (2)--(6)--(3);
                \draw [thick, dashed, black!70] (5)--(7)--(4);
				\draw [thick, black!70] (5)--(4);
                \draw [thick, black!70] (1)--(5);

				\draw[thick, black!70] (1)..controls (3,4) and (8,0)..(1);
				\end{tikzpicture}
\end{center}	
\end{multicols}
\caption{\label{fig4} The graph $G$ (left) and the signed graph $\tilde{G}_D$ (right) with $X_1=K_3$}
\end{figure}
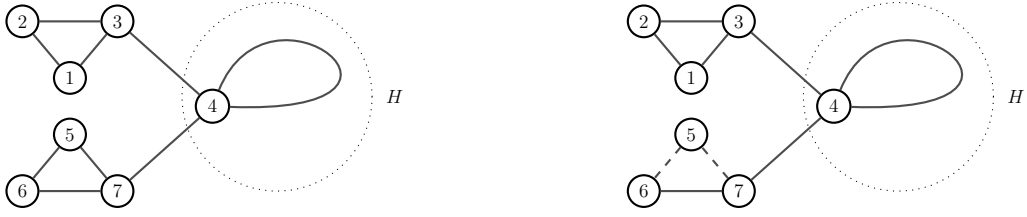

\begin{exa}
Let $G$ be a bounded graph with twin subgraphs $X_1=X_2=K_n$, where $V(X_1)=\{1,\ldots,n\}$ and $V(X_2)=\{n+2,\ldots,2n+1\}$ (see Figure \ref{fig4} for $n=3$).
Consider $\tilde{G}_D$, where $D$ is a diagonal matrix of $\pm 1$'s with $D_1=-D_{n+2}$. Each vertex of $X_1$ is $(1-\frac{2}{n})$-sedentary, and so by Theorem \ref{twins}, $\frac{1}{\sqrt{2}}(\e_j-\e_k)$ for $j\in\{1,\ldots,n-1\}$ and $k\in\{n+2,\ldots,2n\}$ is $(1-\frac{2}{n})$-sedentary in $G$, while $\frac{1}{\sqrt{2}}(\e_1+\e_{n+2})$ is $(1-\frac{2}{n})$-sedentary in $\tilde{G}_D$.
\end{exa}

The next result is analogous to Corollary \ref{mainres1} and again follows from Theorem \ref{3t3}.

\begin{figure}[h!]\label{fig:small} 
	\begin{center}
		\begin{tikzpicture}
		\tikzset{enclosed/.style={draw, circle, inner sep=0pt, minimum size=.3cm}}	   
	   \node[enclosed, label={left, yshift=0cm: $u$}] (w_1) at (0.4,1.9) {};
	    \node[enclosed, label={left, yshift=0cm: $v$}] (w_2) at (0.4,3.1) {};
		\node[enclosed] (w_3) at (1.5,2.5) {};
		\node[enclosed] (w_4) at (3,2.5) {};
		\node[enclosed] (w_6) at (5,2.5) {};
		\draw (w_1) -- (w_3);
		\draw (w_2) -- (w_3);
		\draw (w_4) -- (w_3);
		\draw (w_4) -- node[below] {$\underbrace{}_{n-2\ \text{vertices}}$} (w_6)[dashed];
		\end{tikzpicture}
	\end{center}
	\caption{The graph $P_n'$ with twin vertices $u$ and $v$}\label{yay}
\end{figure}
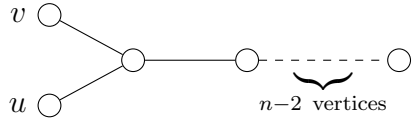

\begin{cor}
\label{plusPi2}
Let $G$ be a bounded graph and $\Pi$ be an equitable partition of $G$ with cell $V_1=\{a,b\}$. 
The following are equivalent.
\begin{enumerate}
\item Vertex $V_1$ is $C$-sedentary in $G/\Pi$.
\item The plus state $\frac{1}{\sqrt2}(\e_a+\e_{b})$ is $C$-sedentary in $G$.
\item The pair state $\frac{1}{\sqrt2}(\e_a-\e_{b})$ is $C$-sedentary in $\tilde{G}_D$, where $D$ is a diagonal matrix of $\pm 1$'s with $D_{a}=-D_{b}$.
\end{enumerate}
\end{cor}

Applying Corollary \ref{plusPi2} to blow-up graphs gives us the following result.

\begin{cor}
\label{sedbu}
Let $H$ be a graph and $X:=\overset{2}{\uplus}~H$. The following are equivalent.
\begin{enumerate}
\item Vertex $a$ is $C$-sedentary in $H$.
\item The plus state $\frac{1}{\sqrt2}(\e_{(0,a)}+\e_{(1,a)})$ is $C$-sedentary in $X$.
\item The pair state $\frac{1}{\sqrt2}(\e_{(0,a)}-\e_{(1,a)})$ is $C$-sedentary in $\tilde{X}$, where $A(\tilde{X})=DA(X)D$ for some diagonal matrix $D$ of $\pm 1$'s such that $D_{(1,a)}=-D_{(0,a)}.$
\end{enumerate}
\end{cor}

Again, using the fact that each vertex $a$ of $K_n$ is $C$-sedentary for all $n\geq 3$, where $C=\frac{n-2}{n}$, we get the following consequence of Corollary \ref{sedbu}.

\begin{exa}
The plus state $\frac{1}{\sqrt2}(\e_{(0,a)}+\e_{(1,a)})$ is $C$-sedentary in $\overset{2}{\uplus}~K_n$ for all $a\in V(K_n)$.
\end{exa}

Let $n\geq 3$ be odd and consider the family $P_n'$ of graphs obtained from $P_n$ by adding vertex $v$ that is twins with vertex $u=1$ of $P_n$ (see Figure \ref{yay}). In \cite{monterde2023new}, Monterde showed that $u$ and $v$ are sedentary vertices. 
Applying Corollary \ref{3c5} with $X_1=X_2=P_n'$ as twin subgraphs yields the following two results, one of which is analogous to Theorem \ref{inftreesmaxdeg3} and Corollary \ref{inftreesmaxdeg31}, and another one analogous to Theorem \ref{infsignedplus} and Corollary \ref{treemaxdeg3}, respectively.

\begin{thm}
For each $k\geq 3$, there are infinitely many connected unweighted graphs with maximum degree $k$ containing a sedentary pair state. In particular, there are infinitely many unweighted trees with maximum degree three and a sedentary pair state.
\end{thm}

\begin{thm}
For each $k\geq 3$, there are infinitely many connected unweighted signed graphs $\hat{G}$ such that the underlying graphs have maximum degree $k$, $\hat{G}$ has exactly one edge of negative weight, and $\hat{G}$ contains a sedentary plus state. In particular, there are infinitely many unweighted signed trees $\hat{T}$ such that the underlying graphs have maximum degree three, $\hat{T}$ has exactly one negative edge weight and $\hat{T}$ contains a sedentary plus state.
\end{thm}

Lastly, adapting similar proofs to that of Theorem \ref{almostallplanar} and \ref{almostalltrees} yields the following result.

\begin{thm}
Almost all labelled connected unweighted planar graphs contain a sedentary pair state. 
Moreover, almost all labelled connected unweighted planar graphs can be assigned a single negative edge so that the resulting signed graphs contain a sedentary plus state. This statement holds if we replace `labelled connected planar graphs' with `trees'.
\end{thm}


\section{Cayley Graphs}\label{sec:cay}

Let $(\Gamma,+)$ be a finite abelian group and consider a connection set $S\subset\Gamma$, without containing the identity, satisfying  $\cb{-s : s \in S} = S.$ A Cayley graph over $\Gamma$ with the connection set $S$ is denoted by $\text{Cay}~(\Gamma, S)$ has the vertex set $\Gamma$ where two vertices $a, b \in \Gamma$ are adjacent if and
only if $a - b \in S.$ The cycle $C_n$ on $n$ vertices  is a Cayley graph over $\mathbb{Z}_n$ with connection set $\cb{-1,1}.$ There is no plus state transfer in the cycle $C_6$ \cite[Theorem 6.5]{kim}. However, assigning $-1$ to the edges $\{3,4\}$ and $\{0,5\}$ in $C_6$ results perfect state transfer between $\frac{1}{\sqrt2}(\e_1+\e_5)$ and $\frac{1}{\sqrt2}(\e_2+\e_4).$ Using the characterization of pretty good pair state transfer in cycles \cite[Theorem 9]{pal10}, we obtain an observation on a signed cycle.

\begin{thm}
The cycle $C_n$ with $w\ob{0,n-1}=w\ob{\frac{n}{2},\frac{n}{2}+1}=-1$ exhibits pretty good state transfer between $\frac{1}{\sqrt2}(\e_a+\e_{n-a})$ and $\frac{1}{\sqrt2}(\e_{\frac{n}{2}-a}+\e_{\frac{n}{2}+a}),$ whenever $0<a<\frac{n}{2}$ with $a\neq\frac{n}{4}$ and either:
    \begin{enumerate}
    \item $n = 2^{k},$ where $k$ is a positive integer greater than two or,
    \item $n=2^k p,$ where $k$ is a positive integer and $p$ is an odd prime, and $a=2^{k-2}r,$ for some positive integer $r.$
    \end{enumerate}   
\end{thm}

Another way to view a signed graph $\tilde{G}$ is as two edge-disjoint spanning subgraphs of the underlying unsigned graph $G$ that is $\tilde{G}=G^+\cup G^-,$ where the edges of $G^+$ and $G^-$ are signed with $1$ and $-1$, respectively. The adjacency matrix of $\tilde{G}$ is given by
\begin{center}
$A\ob{\tilde{G}}=A\ob{G^+} - A\ob{G^-}$.
\end{center}
Such a decomposition of the adjacency matrix of the signed graph into positive and negative parts gives rise to the following result. The proof of the following result is omitted as it is similar to that of
perfect state transfer between vertex states in \cite[Theorem 2]{brow}. 

\begin{thm}\label{3t10}
Let $H$ and $K$ be edge-disjoint spanning subgraphs of a graph $G$ such that $G=H\cup K.$ If $H$ has perfect state transfer between $\frac{1}{\sqrt2}(\e_a\pm\e_b)$ and $\frac{1}{\sqrt2}(\e_c\pm\e_d)$ at $\tau$ and $K$ is periodic at $\frac{1}{\sqrt2}(\e_a\pm\e_b)$ at $\tau$, then the signed graph $\tilde{G}=H^+\cup K^-$ has perfect state transfer between $\frac{1}{\sqrt2}(\e_a\pm\e_b)$ and $\frac{1}{\sqrt2}(\e_c\pm\e_d)$ at $\tau$ provided $A(H)$ and $A(K)$ commute. 
\end{thm}

We also need the following result.

\begin{prop}\label{3p2}\cite{pal6}
If $S_1$ and $S_2$ are symmetric subsets of an abelian group $\Gamma$ then adjacency matrices of the Cayley graphs $\text{Cay}\ob{\Gamma, S_1}$ and $\text{Cay}\ob{\Gamma , S_2}$ commute.
\end{prop}

We now construct signed Cayley graphs with pair state and plus state transfer. 

\begin{exa}
Let $\Gamma$ be an abelian group of order $4n.$ Consider $\text{Cay}\ob{\mathbb{Z}_6\times \Gamma, S_1}$ and $\text{Cay}\ob{\mathbb{Z}_6\times \Gamma, S_2}$, with $S_1=\cb{(\pm 1,0)}$and $S_2=\cb{(0,j)}$ for all $j\in \Gamma \setminus \cb{0}.$ The graph $\text{Cay}\ob{\mathbb{Z}_6\times \Gamma, S_1}$ is formed by $4n$ number of disjoint cycles $C_6,$ while $\text{Cay}\ob{\mathbb{Z}_6\times \Gamma, S_2}$ is the graph formed by six complete graphs of order $4n.$ Now, the graph $\text{Cay}\ob{\mathbb{Z}_6\times \Gamma, S_1\cup S_2}$ is a connected graph. Since the complete graph $K_{4n}$ is periodic at $\frac{\pi}{2}$ and $C_6$ has pair state transfer between $\frac{1}{\sqrt2}(\e_0-\e_2)$ and $\frac{1}{\sqrt2}(\e_3-\e_5)$ at the same time, then by Theorem \ref{3t10} and Proposition \ref{3p2}, the signed graph $\text{Cay}\ob{\mathbb{Z}_6\times \Gamma, S_1}\cup \text{Cay}\ob{\mathbb{Z}_6\times \Gamma, S_2}^-$ has pair state transfer at $\frac{\pi}{2}$ between $\frac{1}{\sqrt2}(\e_{(0,j)}-\e_{(2,j)})$ and $\frac{1}{\sqrt2}(\e_{(3,j)}-\e_{(5,j)})$ for $j\in\Gamma.$
\end{exa}

\begin{exa}
Let $C$ be a subset of $\mathbb{Z}_2^d$ such that the sum of the elements of $C$ is $0$. Then each vertex in $\text{Cay}\ob{\mathbb{Z}_2^d, C}$ is periodic at $\frac{\pi}{2}$ \cite{che}. Consider $\text{Cay}\ob{\mathbb{Z}_8\times \mathbb{Z}_2^d , S_1}$ and $\text{Cay}\ob{\mathbb{Z}_8\times \mathbb{Z}_2^d, S_2}$, with $S_1=\cb{(\pm 1,0)}$ and $S_2=\cb{(0,j)}$ for all $j\in C.$ The graph $\text{Cay}\ob{\mathbb{Z}_8\times \mathbb{Z}_2^d , S_1}$ is formed by $2^d$ number of disjoint cycles $C_8,$ while $\text{Cay}\ob{\mathbb{Z}_8\times \mathbb{Z}_2^d, S_2}$ is the graph formed by eight cube-like graphs of order $2^d.$ Now, the graph $\text{Cay}\ob{\mathbb{Z}_8\times \mathbb{Z}_2^d, S_1\cup S_2}$ is connected graph. Since the cube-like graph $\text{Cay}\ob{\mathbb{Z}_2^d, C}$ is periodic at $\frac{\pi}{2}$ and $C_8$ has plus state transfer between $\frac{1}{\sqrt2}(\e_0+\e_4)$ and $\frac{1}{\sqrt2}(\e_2+\e_6)$ at the same time, Theorem \ref{3t10} and Proposition \ref{3p2} implies that the signed graph $\text{Cay}\ob{\mathbb{Z}_8\times \mathbb{Z}_2^d, S_1}\cup \text{Cay}\ob{\mathbb{Z}_8\times \mathbb{Z}_2^d, S_2}^-$ has plus state transfer at $\frac{\pi}{2}$ between $\frac{1}{\sqrt2}(\e_{(0,j)}+\e_{(4,j)})$ and $\frac{1}{\sqrt2}(\e_{(2,j)}+\e_{(6,j)})$ for $j\in\mathbb{Z}_2^d.$
\end{exa}

\section{Graphs with tails}
\label{gwt}

Let $G$ and $H$ be graphs. The graph obtained by identifying a vertex of $G$ with a vertex of $H$, say $V(G)\cap V(H)=\{u\}$, is called the \textit{1-sum} of $G$ and $H$ at vertex $u$. We say that $G$ is a \textit{graph with a tail} if $G$ is a 1-sum of a graph and a path, where the 1-sum is taken at the pendant vertex of the path. Let $\mathcal{Y}=\{H_j(r_j):j\in\{1,\ldots,|V(G)|\}\}$ be a collection of rooted graphs, where each $H_j$ is rooted at vertex $r_j$. The \textit{rooted product} of $G$ with $\mathcal{Y}$, denoted $G^\mathcal{Y}$, is obtained from the 1-sum of $G$ with $H_j$ by identifying vertex $j$ of $G$ with vertex $r_j$ in $H_j$, for each $j$.

In this section, we apply our results to graphs with tails. Throughout, we let $P_{\infty}$ denote the infinite path. The following result is a strengthening of \cite[Theorem 15]{bernard2025quantum}.

\begin{cor}
\label{aak}
Let $X_1$ be a bounded graph and $G=P_3^\mathcal{Y}$ where $\mathcal{Y}=\{X_1,P_{\infty},X_2\}$ and $f:X_1\rightarrow X_2$ is an isomorphism. Then $X_1$ has perfect state transfer between vertices $u$ and $v$ if and only if $G$ has perfect state transfer between $\frac{1}{\sqrt{2}}(\e_u-\e_{f(u)})$ and $\frac{1}{\sqrt{2}}(\e_v-\e_{f(v)})$.
\end{cor}

\begin{proof}
Note that $X_1$ and $X_2$ are twin subgraphs in $P_3^\mathcal{Y}$. As there are no edges between $X_1$ and $X_2$, we get $A'=0$.  Applying Corollary \ref{mainres}(1-2) yields the desired result.
\end{proof}

We now construct graphs with tails with pair PST that do not arise from rooted products of $P_3$. Let $X$ be the Cartesian product of two copies of $P_3$. The \textit{fly-swatter graph} $G_n$ is obtained from a 1-sum of $X$ and $P_n$ at vertex $u$, where $u$ is a degree two vertex of $X$ and a leaf in $P_n$ (see Figure \ref{yay5}). The following generalizes \cite[Example 11]{bernard2025quantum}.

\begin{cor}
\label{fsg}
For $n\in\mathbb{Z}^+\cup{\infty}$, let $G_n$ be the fly-swatter graph. Then perfect state transfer occurs between $\frac{1}{\sqrt{2}}(\e_1-\e_7)$ and $\frac{1}{\sqrt{2}}(\e_3-\e_5)$ in $G_n$ for all $n\in\mathbb{Z}^+\cup{\infty}$ at $\frac{\pi}{\sqrt{2}}$. Moreover, if $\hat{G}_n$ and $G'_n$ are signed graphs with adjacency matrices $DA(G_n)D$ and $D'A(G_n)D'$ respectively, where $D$ and $D'$ are diagonal matrices of $\pm 1$'s such that $D_1=-D_7$, $D_3=-D_5$, $D'_1=-D'_7$, and $D'_3=D'_5$, then perfect state transfer occurs between $\frac{1}{\sqrt{2}}(\e_1+\e_7)$ and $\frac{1}{\sqrt{2}}(\e_3+\e_5)$ in $\hat{G}_n$ and between $\frac{1}{\sqrt{2}}(\e_1+\e_7)$ and $\frac{1}{\sqrt{2}}(\e_3-\e_5)$ in $G_n'$ for all $n\in\mathbb{Z}^+\cup{\infty}$ at $\frac{\pi}{\sqrt{2}}$.
\end{cor}

\begin{proof}
Note that $G_n$ has twin subgraphs $X_1=X_2=P_3$, where the subgraph induced by $V(G)\backslash (V(X_1)\cup V(X_2))$ is a union of $P_n$ and two isolated vertices. As PST occurs between vertices $1$ and $3$ in $X_1$ and $A'=0$, Corollary \ref{mainres}(2) yields the desired result.
\end{proof}

\begin{figure}[h!]
\begin{center}
 \begin{tikzpicture}[scale=.5,auto=left]
                       \tikzstyle{every node}=[circle, thick, fill=white, scale=0.6]
                       \node[draw,minimum size=0.7cm, inner sep=0 pt] (1) at (-6.5, 1.2) {$1$};
                \node[draw,minimum size=0.7cm, inner sep=0 pt] (2) at (-3.5, 2.9) {$2$};
		        \node[draw,minimum size=0.7cm, inner sep=0 pt] (3) at (-0.5, 1.2) {$3$};
                \node[draw,minimum size=0.7cm, inner sep=0 pt] (4) at (2.4, -0.5) {$4$};	
                \node[draw,minimum size=0.7cm, inner sep=0 pt] (8) at (-9.4, -0.5) {$8$};	
                \node[draw,minimum size=0.7cm, inner sep=0 pt] (9) at (-3.5, -0.5) {$9$};
		        \node[draw,minimum size=0.7cm, inner sep=0 pt] (5) at (-0.5, -2.2) {$5$};
		     \node[draw,minimum size=0.7cm, inner sep=0 pt] (6) at (-3.5, -3.9) {$6$};  
	       \node[draw,minimum size=0.7cm, inner sep=0 pt] (7) at (-6.5, -2.2) {$7$};
                \node[draw,minimum size=0.5cm, inner sep=0 pt] (10) at (5.4, -0.5) {};	
                \node[draw,minimum size=0.5cm, inner sep=0 pt] (11) at (8.4, -0.5) {};	
                \node[draw,minimum size=0.5cm, inner sep=0 pt] (12) at (11.4, -0.5) {};	

                \node at (-8, 3) {$X_1$};
			\node at (-8,-4) {$X_2$};
                \draw [dotted] (-3.5,1.8) ellipse (4.3cm and 1.7cm);
                \draw [dotted] (-3.5,-2.8) ellipse (4.3cm and 1.7cm);
				\draw [thick, black!70] (9)--(1)--(2)--(3)--(4)--(5)--(6)--(7)--(9);
                \draw [thick, black!70] (4)--(10)--(11)--(12);
                \draw [thick, black!70] (1)--(8)--(7);
                \draw [thick, black!70] (3)--(9)--(5);
\end{tikzpicture}
\caption{The fly-swatter graph $G_4$}\label{yay5}
\end{center}
\end{figure}
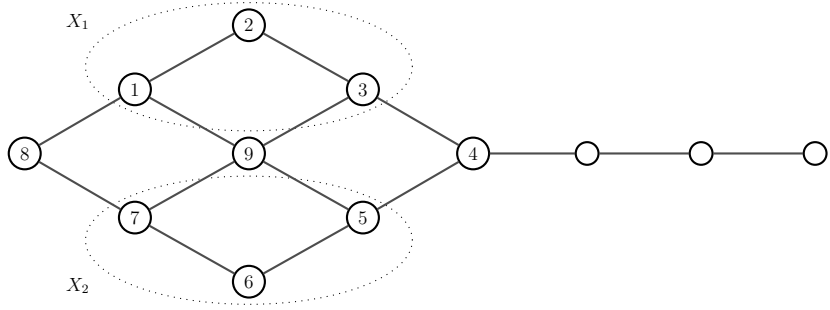


Let $V(C_n)=\mathbb{Z}_n$ and $E(C_n)=\{(j,j+1):j\in\mathbb{Z}_n\}$. For $p\geq 3$, let $H_{2p}$ be the resulting graph after adding the matching $M$ to $C_{2p}$, where $M=\varnothing$ if $p=3,4$, and for $p\geq 5$, 
\begin{equation*}
M=\begin{cases}
    \left\{\{\frac{p-3}{2},\frac{3p+1}{2}\},\{\frac{p-1}{2},\frac{3(p+1)}{2}\},\{\frac{p+1}{2},\frac{3(p-1)}{2}\},\{\frac{p+3}{2},\frac{3p-1}{2}\}\right\} & \text{if $p$ is odd}\\
    \left\{\{\frac{p}{2}-2,\frac{3p}{2}+1\},\{\frac{p}{2}-1,\frac{3p}{2}+2\},\{\frac{p}{2}+1,\frac{3p}{2}-2\},\{\frac{p}{2}+2,\frac{3p}{2}-1\}\right\} & \text{otherwise.}
  \end{cases}
\end{equation*}

\begin{cor}
Let $G_{p,n}$ be the graph obtained from a 1-sum of $H_{2p}$ and a path $P_{n}$ at vertex $p$. The following hold for all integers $p\geq 3$ and for all $n\in\mathbb{Z}^+\cup{\infty}$.
\begin{enumerate}
\item Perfect state transfer occurs between $\frac{1}{\sqrt{2}}(\e_{\lceil\frac{p}{2}-1\rceil}-\e_{\lfloor\frac{3p}{2}+1\rfloor})$ and $\frac{1}{\sqrt{2}}(\e_{\lfloor\frac{p}{2}+1\rfloor}-\e_{\lceil\frac{3p}{2}-1\rceil})$ in $G_{p,n}$ at $\tau=\frac{\pi}{2}$ whenever $p$ is odd and $\tau=\frac{\pi}{\sqrt{2}}$ otherwise.
\item Let $\hat{G}_{p,n}$ and $G'_{p,n}$ be signed graphs with $A(\hat{G}_{p,n})=DA(G_n)D$ and $A(G'_{p,n})=D'A(G_n)D'$, where $D,D'$ are diagonal matrices of $\pm 1$'s such that $D_{\lceil\frac{p}{2}-1\rceil}=-D_{\lfloor\frac{3p}{2}+1\rfloor}$, $D_{\lfloor\frac{p}{2}+1\rfloor}=-D_{\lceil\frac{3p}{2}-1\rceil}$,  $D'_{\lceil\frac{p}{2}-1\rceil}=-D'_{\lfloor\frac{3p}{2}+1\rfloor}$, and $D'_{\lfloor\frac{p}{2}+1\rfloor}=D'_{\lceil\frac{3p}{2}-1\rceil}$. Then perfect state transfer occurs between $\frac{1}{\sqrt{2}}(\e_{\lceil\frac{p}{2}-1\rceil}+\e_{\lfloor\frac{3p}{2}+1\rfloor})$ and $\frac{1}{\sqrt{2}}(\e_{\lfloor\frac{p}{2}+1\rfloor}+\e_{\lceil\frac{3p}{2}-1\rceil})$ in $\hat{G}_{p,n}$ and between $\frac{1}{\sqrt{2}}(\e_{\lceil\frac{p}{2}-1\rceil}+\e_{\lfloor\frac{3p}{2}+1\rfloor})$ and $\frac{1}{\sqrt{2}}(\e_{\lfloor\frac{p}{2}+1\rfloor}-\e_{\lceil\frac{3p}{2}-1\rceil})$ in $G_{p,n}'$ at $\tau$ given in 1.
\end{enumerate}
\end{cor}

\begin{figure}[h!]
\begin{multicols}{2}
\begin{center}
 \begin{tikzpicture}[scale=.5,auto=left]
                       \tikzstyle{every node}=[circle, thick, fill=white, scale=0.6]
                       \node[draw,minimum size=0.7cm, inner sep=0 pt] (1) at (-6.5, 1.3) {$1$};
		        \node[draw,minimum size=0.7cm, inner sep=0 pt] (3) at (-3, 1.3) {$2$};
                \node[draw,minimum size=0.7cm, inner sep=0 pt] (4) at (-1, -1) {$3$};	
                \node[draw,minimum size=0.7cm, inner sep=0 pt] (8) at (-8.5, -1) {$0$};	
		        \node[draw,minimum size=0.7cm, inner sep=0 pt] (5) at (-3, -3.3) {$4$}; 
	       \node[draw,minimum size=0.7cm, inner sep=0 pt] (7) at (-6.5, -3.3) {$5$};
                \node[draw,minimum size=0.5cm, inner sep=0 pt] (10) at (1, -1) {};	
                \node[draw,minimum size=0.5cm, inner sep=0 pt] (11) at (3, -1) {};		
                \node at (-8, 1.5) {$X_1$};
			\node at (-8,-3.5) {$X_2$};
                \draw [dotted] (-4.75,1.4) ellipse (2.6cm and 1.2cm);
                \draw [dotted] (-4.75,-3.4) ellipse (2.6cm and 1.2cm);
				\draw [thick, black!70] (1)--(3)--(4)--(5)--(7);
                \draw [thick, black!70] (4)--(10)--(11);
                \draw [thick, black!70] (1)--(8)--(7);
\end{tikzpicture}
\begin{tikzpicture}[scale=.5,auto=left]
                       \tikzstyle{every node}=[circle, thick, fill=white, scale=0.6]
                       \node[draw,minimum size=0.7cm, inner sep=0 pt] (1) at (-6.5, 2) {$1$};
                \node[draw,minimum size=0.7cm, inner sep=0 pt] (2) at (-3.5, 2.9) {$2$};
		        \node[draw,minimum size=0.7cm, inner sep=0 pt] (3) at (-0.5, 2) {$3$};
                \node[draw,minimum size=0.7cm, inner sep=0 pt] (4) at (1, -0.5) {$4$};	
                \node[draw,minimum size=0.7cm, inner sep=0 pt] (8) at (-8, -0.5) {$0$};	
		        \node[draw,minimum size=0.7cm, inner sep=0 pt] (5) at (-0.5, -3) {$5$};
		     \node[draw,minimum size=0.7cm, inner sep=0 pt] (6) at (-3.5, -3.9) {$6$};  
	       \node[draw,minimum size=0.7cm, inner sep=0 pt] (7) at (-6.5, -3) {$7$};
                \node[draw,minimum size=0.5cm, inner sep=0 pt] (10) at (3, -0.5) {};	
                \node[draw,minimum size=0.5cm, inner sep=0 pt] (11) at (5, -0.5) {};	
                
                \node at (-8, 3) {$X_1$};
			\node at (-8,-4) {$X_2$};
                \draw [dotted] (-3.5, 1.9) ellipse (3.7cm and 1.7cm);
                \draw [dotted] (-3.5, -2.9) ellipse (3.7cm and 1.7cm);
				\draw [thick, black!70] (1)--(2)--(3)--(4)--(5)--(6)--(7);
                \draw [thick, black!70] (4)--(10)--(11);
                \draw [thick, black!70] (1)--(8)--(7);
\end{tikzpicture}
\end{center}
\end{multicols}
\caption{The graph $G_{3,3}$ (left) with PST between $\frac{1}{\sqrt2}(\e_1-\e_5)$ and $\frac{1}{\sqrt2}(\e_2-\e_4)$ at $\frac{\pi}{2}$, and $G_{4,3}$ (right) with PST between $\frac{1}{\sqrt2}(\e_1-\e_7)$ and $\frac{1}{\sqrt2}(\e_3-\e_5)$ at $\frac{\pi}{\sqrt{2}}$}\label{yay6}
\end{figure}
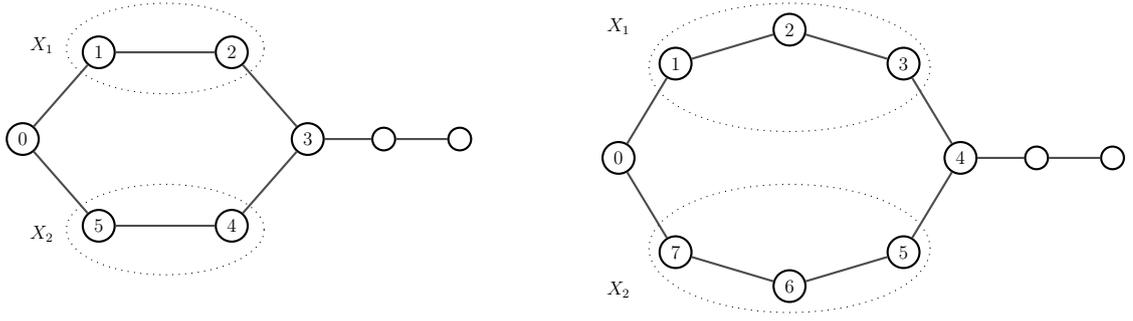

\begin{proof}
If $p$ is odd, then our assumption on $H_{2p}$ gives $G_{p,n}$ with twin subgraphs $X_1=X_2=K_2$ where $V(X_1)=\{\frac{p-1}{2},\frac{p+1}{2}\}$, $V(X_2)=\{\frac{3p+1}{2},\frac{3p-1}{2}\}$ and the subgraph induced by $V(G)\backslash (V(X_1)\cup V(X_2))$ is a union of $P_{p-2}$ and a one-sum of $P_{p-2}$ and $P_n$ at vertex $u$ which is the middle vertex of $P_{p-2}$ and a leaf of $P_n$. If $p$ is even, then $G_{p,n}$ has twin subgraphs $X_1=X_2=P_3$ where $V(X_1)=\{\frac{p}{2}-1,\frac{p}{2},\frac{p}{2}+1\}$, $V(X_2)=\{\frac{3p}{2}+1,\frac{3p}{2},\frac{3p}{2}-1\}$, and the subgraph induced by $V(G)\backslash (V(X_1)\cup V(X_2))$ is a union of $P_{p-3}$ and a one-sum of $P_{p-3}$ and $P_n$ at vertex $u$ which is the middle vertex of $P_{p-3}$ and a leaf of $P_n$. 
Since $A'=0$ in both cases, and $K_2$ and $P_3$ admit PST between end vertices at $\frac{\pi}{2}$ and $\frac{\pi}{\sqrt{2}}$, respectively, Corollary \ref{mainres}(2) yields 1, while Corollary \ref{mainres}(3-4) yields 2.
\end{proof}


\section{Open questions}
\label{oq}

In \cite[Theorem 5.6]{godd}, it is shown that no vertex state in a bounded infinite graph admits periodicity (which in turn implies that vertex PST does not occur in bounded infinite graphs). Thus, it is natural to wonder whether infinite graphs in general do not admit PST. In \cite{godd}, examples of locally finite graphs that are not bounded possessing periodic vertices are given, but it is unknown whether they are involved in perfect state transfer.

We also constructed signed graphs with exactly one negative edge weight having PST between a pair state and a plus state. It would be interesting to find a characterization of PST between a plus state and a pair state in undirected connected unweighted graphs.

Finally, recall that the number of edges in a planar graph is linear in its number of vertices. This relative sparsity makes it appealing to characterize vertex PST in planar graphs. 
In particular, are there infinite families of planar graphs with vertex PST?

\section*{Disclosure statement}
No potential conflict of interest was reported by the author(s).
 
\section*{Acknowledgements} 
C.\ Godsil is supported by NSERC grant number RGPIN-9439. S.\ Kirkland is supported by NSERC grant number RGPIN-2025-05547. S.\ Mohapatra is supported by the Department of Science and Technology (INSPIRE: IF210209). H.\ Monterde is supported by the University of Manitoba Faculty of Science and Faculty of Graduate Studies.


\bibliographystyle{abbrv}
\bibliography{references}

\end{document}